\documentclass[leqno, a4paper, 12pt]{article}

\usepackage{amssymb, amscd, amsthm}

\usepackage[all]{xy}

\theoremstyle{plain}
\newtheorem*{thm}{Theorem}
\newtheorem{theorem}{Theorem}[section]
\newtheorem{lemma}[theorem]{Lemma}
\newtheorem{proposition}[theorem]{Proposition}
\newtheorem{corollary}[theorem]{Corollary}

\theoremstyle{definition}

\newtheorem{remark}[theorem]{Remark}
\newtheorem{remarks}[theorem]{Remarks}

\newtheorem{examples}[theorem]{Examples}

\newcommand\bD{{\mathbb D}}
\newcommand\bF{{\mathbb F}}
\newcommand\bG{{\mathbb G}}

\newcommand\bQ{{\mathbb Q}}

\newcommand\bZ{{\mathbb Z}}

\newcommand\cA{{\mathcal A}}

\newcommand\cC{{\mathcal C}}
\newcommand\cD{{\mathcal D}}

\newcommand\cF{{\mathcal F}}

\newcommand\cI{{\mathcal I}}
\newcommand\cL{{\mathcal L}}
\newcommand\cM{{\mathcal M}}

\newcommand\cO{{\mathcal O}}
\newcommand\cP{{\mathcal P}}

\newcommand\cS{{\mathcal S}}
\newcommand\cT{{\mathcal T}}
\newcommand\cU{{\mathcal U}}
\newcommand\cV{{\mathcal V}}

\newcommand\uf{\underline{f}}
\newcommand\ug{\underline{g}}
\newcommand\uh{\underline{h}}
\newcommand\uu{\underline{u}}
\newcommand\uv{\underline{v}}
\newcommand\uw{\underline{w}}
\newcommand\ugamma{\underline{\gamma}}
\newcommand\uvarphi{\underline{\varphi}}
\newcommand\upsi{\underline{\psi}}

\newcommand\uA{\underline{A}}
\newcommand\uD{\underline{D}}
\newcommand\uE{\underline{E}}
\newcommand\uFC{\underline{FC}}
\newcommand\uR{\underline{R}}
\newcommand\uT{\underline{T}}
\newcommand\uU{\underline{U}}
\newcommand\uX{\underline{X}}

\newcommand\uotimes{\underline{\otimes}}

\newcommand\ucA{\underline{\mathcal A}}
\newcommand\ucC{\underline{\mathcal C}}
\newcommand\ucD{\underline{\mathcal D}}
\newcommand\ucE{\underline{\mathcal E}}
\newcommand\ucL{\underline{\mathcal L}}
\newcommand\ucS{\underline{\mathcal S}}
\newcommand\ucT{\underline{\mathcal T}}
\newcommand\ucU{\underline{\mathcal U}}
\newcommand\ucV{\underline{\mathcal V}}

\newcommand\wA{\widehat{A}}
\newcommand\wB{\widehat{B}}

\newcommand\car{{\rm char}}

\newcommand\diag{{\rm diag}}

\newcommand\hd{{\rm hd}}
\newcommand\id{{\rm id}}

\renewcommand\ker{{\rm ker}}

\newcommand\red{{\rm red}}

\newcommand\R{{\rm R}}

\newcommand\Coker{{\rm Coker}}

\newcommand\End{{\rm End}}
\newcommand\Ext{{\rm Ext}}

\newcommand\GL{{\rm GL}}

\newcommand\Hom{{\rm Hom}}
\renewcommand\Im{{\rm Im}}

\newcommand\Ker{{\rm Ker}}
\newcommand\Lie{{\rm Lie}}

\newcommand\Rep{{\rm Rep}}

\newcommand\Spec{{\rm Spec}}

\title{Commutative algebraic groups up to isogeny}

\author{Michel Brion}

\date{}

\begin{document}

\maketitle

\begin{abstract}
Consider the abelian category $\cC_k$ of commutative group schemes of finite
type over a field $k$. By results of Serre and Oort, $\cC_k$ has homological 
dimension $1$ (resp.~$2$) if $k$ is algebraically closed of characteristic 
$0$ (resp.~positive). In this article, we explore the abelian category 
of commutative algebraic groups up to isogeny, defined as the quotient 
of $\cC_k$ by the full subcategory $\cF_k$ of finite $k$-group schemes. 
We show that $\cC_k/\cF_k$ has homological dimension $1$, and we determine 
its projective or injective objects. We also obtain structure results 
for $\cC_k/\cF_k$, which take a simpler form in positive characteristics. 
\end{abstract}

\tableofcontents

\section{Introduction}
\label{sec:i}

There has been much recent progress on the structure of algebraic
groups over an arbitrary field; in particular, on  the classification of 
pseudo-reductive groups (see \cite{CGP, CP}). Yet commutative
algebraic groups over an imperfect field remain somewhat mysterious,
e.g., extensions with unipotent quotients are largely unknown;
see \cite{Totaro} for interesting results, examples, and questions.

In this article, we explore the category of commutative 
algebraic groups up to isogeny, in which the problems raised
by imperfect fields become tractable; this yields rather simple and 
uniform structure results.
 
More specifically, denote by $\cC_k$ the category with objects 
the group schemes of finite type over the ground field $k$, 
and with morphisms, the homomorphisms of $k$-group schemes 
(all group schemes under consideration will be assumed commutative). 
By a result of Grothendieck (see \cite[VIA, Thm.~5.4.2]{SGA3}), 
$\cC_k$ is an abelian category. We define the category of 
`algebraic groups up to isogeny' as the quotient category of 
$\cC_k$ by the Serre subcategory of finite group schemes; then 
$\cC_k/\cF_k$ is obtained from $\cC_k$ by inverting all isogenies, 
i.e., all morphisms with finite kernel and cokernel.

It will be easier to deal with the full subcategory $\ucC_k$ of 
$\cC_k/\cF_k$ with objects the smooth connected algebraic groups, 
since these categories turn out to be equivalent, and morphisms 
in $\ucC_k$ admit a simpler description.

As a motivation for considering the `isogeny category' $\ucC_k$,
note that some natural constructions involving algebraic groups 
are only exact up to isogeny; for example, the formations of the 
maximal torus or of the largest abelian variety quotient, both of 
which are not exact in $\cC_k$. Also, some structure theorems for 
algebraic groups take on a simpler form when reformulated up to isogeny. 
A classical example is the Poincar\'e complete reducibility theorem, 
which is equivalent to the semi-simplicity of the isogeny category 
of abelian varieties, i.e., the full subcategory $\ucA_k$ of $\ucC_k$ 
with objects abelian varieties. Likewise, the isogeny category of tori, 
$\ucT_k$, is semi-simple.

We gather our main results in the following:

\begin{thm}\label{thm:all}

\begin{enumerate}

\item[{\rm (i)}] The category $\ucC_k$ is artinian and noetherian.
Its simple objects are exactly the additive group $\bG_{a,k}$, 
the simple tori, and the simple abelian varieties.

\item[{\rm (ii)}] The product functor
$\ucT_k \times \ucU_k \to \ucL_k$ yields an equivalence of categories,
where $\ucU_k$ (resp.~$\ucL_k$) denotes the isogeny category of
unipotent (resp.~linear) algebraic groups. 

\item[{\rm (iii)}] If $\car(k) > 0$, then the product functor
$\ucS_k \times \ucU_k \to \ucC_k$ yields an equivalence of categories,
where $\ucS_k$ denotes the isogeny category of semi-abelian varieties.
If in addition $k$ is locally finite, then the product functor
$\ucT_k \times \ucA_k \to \ucS_k$ yields an equivalence of categories
as well.

\item[{\rm (iv)}] The base change under any purely inseparable field 
extension $k'$ of $k$ yields an equivalence of categories 
$\ucC_k \to \ucC_{k'}$.

\item[{\rm (v)}]
The homological dimension of $\ucC_k$ is $1$.

\end{enumerate}

\end{thm}

We also describe the projective objects of the category $\ucC_k$ 
(Theorem \ref{thm:proj}) and its injective objects (Theorem 
\ref{thm:inj}). Moreover, in characteristic $0$, we obtain a structure 
result for that category (Proposition \ref{prop:prodzero}), 
which turns out to be more technical than in positive characteristics.

Let us now compare the above statements with known results on $\cC_k$ 
and its full subcategories $\cA_k$ (resp.~$\cT_k$, $\cU_k$, $\cL_k$, 
$\cS_k$) of abelian varieties (resp.~tori, unipotent groups, linear 
groups, semi-abelian varieties). 

About (i) (an easy result, mentioned by Serre in \cite{Serre-II}):
$\cC_k$ is artinian and not noetherian. Also, every 
algebraic group is an iterated extension of `elementary' groups; 
these are the simple objects of $\ucC_k$ and the simple finite
group schemes.

About (ii): denoting by $\cM_k$ the full subcategory of $\cC_k$
with objects the algebraic groups of multiplicative type, 
the product functor $\cM_k \times \cU_k \to \cL_k$ 
yields an equivalence of categories if $k$ is perfect. 
But over an imperfect field, there exist non-zero extensions of 
unipotent groups by tori, and these are only partially understood
(see \cite[\S 9]{Totaro} again; our study of $\ucC_k$ brings no
new insight in these issues).

About (iii): the first assertion follows from recent structure results 
for algebraic groups (see \cite[\S 5]{Brion-II}), together with a 
lifting property for extensions of such groups with finite quotients
(see \cite{Brion} and \cite{Lucchini}). The second assertion is a direct 
consequence of the Weil-Barsotti isomorphism (see e.g. 
\cite[\S III.18]{Oort66}).

About (iv): this is a weak version of a result of Chow on abelian 
varieties, which asserts (in categorical language) that base change 
yields a fully faithful functor $\cA_k \to \cA_{k'}$ for any primary field 
extension $k'$ of $k$ (see \cite{Chow}, and \cite[\S 3]{Conrad} for 
a modern proof).

About (v), the main result of this article: recall that the homological 
dimension of an abelian category $\cD$ is the smallest integer, 
$\hd(\cD)$, such that $\Ext^n_{\cD}(A,B) = 0$ for all objects $A,B$ of 
$\cD$ and all $n > \hd(\cD)$; these $\Ext$ groups are defined as 
equivalence classes of Yoneda extensions. In particular, $\hd(\cD) = 0$ 
if and only if $\cD$ is semi-simple.

It follows from work of Serre (see \cite[10.1 Thm.~1]{Serre-II}
together with \cite[Thm.~3.5]{Oort64}) that $\hd(\cC_k) = 1$ 
if $k$ is algebraically closed of characteristic $0$.
Also, by a result of Oort (see \cite[Thm.~14.1]{Oort66}), $\hd(\cC_k) =2$
if $k$ is algebraically closed of positive characteristic. Building on
these results, Milne determined $\hd(\cC_k)$ when $k$ is perfect (see 
\cite[Thm.~1]{Milne70}); then the homological dimension can be 
arbitrarily large. In the approach of Serre and Oort, 
the desired vanishing of higher extension groups is obtained by
constructing projective resolutions of elementary groups, in the 
category of pro-algebraic groups. The latter category contains $\cC_k$
as a full subcategory, and has enough projectives.

In contrast, to show that $\hd(\ucC_k) = 1$ over an arbitrary field $k$,
we do not need to go to a larger category. We rather observe that
tori are projective objects in $\ucC_k$, and abelian varieties are
injective objects there. This yields the vanishing of all but three
extension groups between simple objects of $\ucC_k$; two of the three
remaining cases are handled directly, and the third one reduces to
the known vanishing of $\Ext^2_{\cC_k}(\bG_{a,k},\bG_{a,k})$
when $k$ is perfect.

When $k$ has characteristic $0$, the fact that $\hd(\cC_k) \leq 1$ 
follows from a similar result for the category of Laumon $1$-motives 
up to isogeny (obtained by Mazzari in \cite[Thm.~2.5]{Mazzari}).
Indeed, $\cC_k$ is equivalent to a Serre subcategory of 
the latter category; moreover, if an abelian category has homological
dimension at most $1$, then the same holds for any Serre subcategory
(as follows e.g. from \cite[\S 3]{Oort64}). Likewise, the fact that 
the category of Deligne $1$-motives up to isogeny has homological dimension 
at most $1$ (due to Orgogozo, see \cite[Prop.~3.2.4]{Orgogozo}) implies the
corresponding assertion for the isogeny category of semi-abelian 
varieties over an arbitrary field.

Abelian categories of homological dimension $1$ are called hereditary.
The most studied hereditary categories consist either of 
finite-dimensional modules over a finite-dimensional hereditary algebra, 
or of coherent sheaves on a weighted projective line (see e.g. 
\cite{Happel}). Such categories are $k$-linear and $\Hom$-finite, 
i.e., all groups of morphisms are vector spaces of finite dimension 
over the ground field $k$. But this seldom holds for the above isogeny 
categories. More specifically, $\ucA_k$ and $\ucT_k$ are both 
$\bQ$-linear and $\Hom$-finite, but not $\ucC_k$ unless $k$ is 
a number field. When $k$ has characteristic $0$, we may view $\ucC_k$ 
as a mixture of $k$-linear and $\bQ$-linear categories. 
This is already displayed by the full subcategory $\ucV_k$ 
with objects the vector extensions of abelian varieties: as shown in 
\S \ref{subsec:veav}, $\ucV_k$ has enough projectives, and these are 
either the unipotent groups ($k$-linear objects), or the vector 
extensions of simple abelian varieties ($\bQ$-linear objects). 

In positive characteristic, one may also consider the quotient
category of $\cC_k$ by the Serre subcategory $\cI_k$ of infinitesimal
group schemes. This yields the abelian category of 
`algebraic groups up to purely inseparable isogeny', which is 
equivalent to that introduced by Serre in \cite{Serre-II}; as a 
consequence, it has homological dimension $1$ if $k$ is algebraically
closed. For any arbitrary field $k$, the category $\cC_k/\cI_k$ is 
again invariant under purely inseparable field extensions; 
its homological properties may be worth investigating.

\medskip

\noindent
{\bf Notation and conventions}. 
We will use the book \cite{DG} as a general reference, especially
for affine algebraic groups, and the expository text \cite{Brion-II} 
for some further results.

Throughout this text, we fix the ground field $k$ and an algebraic 
closure $\bar{k}$; the characteristic of $k$ is denoted by $\car(k)$.
We denote by $k_s$ the separable closure of $k$ in $\bar{k}$, 
and by $\Gamma_k$ the Galois group of $k_s$ over $k$. We say that $k$
is \emph{locally finite}, if it is algebraic over $\bF_p$ for some
prime $p$.

By an \emph{algebraic $k$-group}, we mean a commutative group
scheme $G$ of finite type over $k$; we denote by $G^0$
the neutral component of $G$. The group law of $G$ will be 
denoted additively: $(x,y) \mapsto x + y$. 

By a \emph{$k$-subgroup} of $G$, we mean a closed $k$-subgroup 
scheme. \emph{Morphisms} are understood to be homomorphisms of 
$k$-group schemes. The (scheme-theoretic) image of a morphism
$f: G \to H$ will be denoted by $\Im(f)$ or $f(G)$, and the
(scheme-theoretic) pull-back of a $k$-subgroup $H' \subset H$,
by $G \times_H H'$ or $f^{-1}(H')$.

Recall that an \emph{abelian variety} over $k$ is 
a smooth connected proper algebraic $k$-group. Also, 
recall that a $k$-group scheme $G$ is an affine algebraic $k$-group
if and only if $G$ is isomorphic to a $k$-subgroup of the general 
linear group $\GL_{n,k}$ for some $n$. We will thus call affine
algebraic $k$-groups \emph{linear}. We say that an algebraic 
$k$-group $G$ is \emph{of multiplicative type} if $G$ is
isomorphic to a $k$-subgroup of some $k$-torus. 

To simplify the notation, \emph{we will suppress the mention of the 
ground field} $k$ whenever this yields no confusion. For example, 
the category $\cC_k$ will be denoted by $\cC$, except when we use 
base change by a field extension.

Given an algebraic group $G$ and two subgroups $G_1,G_2$, 
we denote by $G_1 + G_2$ the subgroup of $G$ generated by $G_1$ 
and $G_2$. Thus, $G_1 + G_2$ is the image of the morphism 
$G_1 \times G_2 \to G$, $(x_1,x_2) \mapsto x_1 + x_2$.

An \emph{isogeny} is a morphism with finite kernel and cokernel.
Two algebraic groups $G_1,G_2$ are \emph{isogenous} if they can 
be connected by a chain of isogenies. 

We say that two subgroups $G_1,G_2$ of an algebraic group 
$G$ are \emph{commensurable} if both quotients $G_1/G_1 \cap G_2$ 
and $G_2/G_1 \cap G_2$ are finite; then $G_1$ and $G_2$ are 
isogenous.

Given an algebraic group $G$ and a non-zero integer $n$,
the multiplication by $n$ yields a morphism 
\[ n_G : G \longrightarrow G. \]
We denote its kernel by $G[n]$, and call it the 
\emph{$n$-torsion subgroup}. We say that $G$ is \emph{divisible}
if $n_G$ is an epimorphism for all $n \neq 0$; then $n_G$ is 
an isogeny for all such $n$. When $\car(k) = 0$, the divisible
groups are exactly the connected algebraic groups; when 
$\car(k) = p > 0$, they are just the \emph{semi-abelian varieties},
that is, the extensions of abelian varieties by tori
(see e.g. \cite[Thm.~5.6.3]{Brion-II} for the latter result). 

Still assuming that $\car(k) = p > 0$, we say that an algebraic
group $G$ is a \emph{$p$-group} if $p^n_G = 0$ for $n \gg 0$. 
Examples of $p$-groups include the unipotent groups and the 
connected finite algebraic groups, also called \emph{infinitesimal}.

\section{Structure of algebraic groups}
\label{sec:scag}

\subsection{Preliminary results}
\label{subsec:pr}

We will use repeatedly the following simple observation:

\begin{lemma}\label{lem:epi}
Let $G$ be a smooth connected algebraic group.

\begin{enumerate}

\item[{\rm (i)}] If $G'$ is a subgroup of $G$ such that
$G/G'$ is finite, then $G' = G$.

\item[{\rm (ii)}] Any isogeny $f: H \to G$ is an epimorphism.

\end{enumerate}

\end{lemma} 

\begin{proof}
(i) The quotient $G/G'$ is smooth, connected and finite,
hence zero.

(ii) This follows from (i) applied to $\Im(f) \subset H$.
\end{proof}

The following lifting result for finite quotients will also 
be frequently used:

\begin{lemma}\label{lem:fin}
Let $G$ be an algebraic group, and $H$ a subgroup such
that $G/H$ is finite. 

\begin{enumerate}

\item[{\rm (i)}] There exists a finite subgroup
$F \subset G$ such that $G = H + F$.

\item[{\rm (ii)}] If $G/H$ is infinitesimal (resp.~a finite $p$-group),
then $F$ may be chosen infinitesimal (resp.~a finite $p$-group) as well.

\end{enumerate}

\end{lemma}

\begin{proof}
(i) This is a special case of \cite[Thm.~1.1]{Brion}.

(ii) Assume $G/H$ infinitesimal. Then the quotient 
$G/(H + F^0)$ is infinitesimal (as a quotient of $G/H$) and 
\'etale (as a quotient of $F/F^0$), hence zero. Thus, we
may replace $F$ with $F^0$, an infinitesimal subgroup.

Next, assume that $G/H$ is a finite $p$-group. Denote by 
$F[p^{\infty}]$ the largest $p$-subgroup of $F$. Then the quotient
$G/(H + F[p^{\infty}])$ is a finite $p$-group and is killed by the
order of $F/F[p^{\infty}]$. Since the latter order is prime to $p$,
we obtain $G/(H + F[p^{\infty}]) = 0$. Thus, we may replace $F$ 
with $F[p^{\infty}]$.  
\end{proof}

Next, we recall a version of a theorem of Chevalley:

\begin{theorem}\label{thm:che}

\begin{enumerate}

\item[{\rm (i)}]
Every algebraic group $G$ contains a linear subgroup $L$
such that $G/L$ is an abelian variety. Moreover, $L$ is 
unique up to commensurability in $G$, and $G/L$ is unique
up to isogeny. 

\item[{\rm (ii)}]
If $G$ is connected, then there exists a smallest such
subgroup, $L = L(G)$, and this subgroup is connected.

\item[{\rm (iii)}]
If in addition $G$ is smooth, then every morphism from 
$G$ to an abelian variety factors uniquely through 
the quotient map $G \to G/L(G)$.

\end{enumerate}

\end{theorem}

\begin{proof}
The assertion (ii) follows from \cite[Lem.~IX 2.7]{Raynaud}
(see also \cite[9.2 Thm.~1]{BLR}).

To prove (i), note that $G$ contains a finite subgroup
$F$ such that $G/F$ is connected (as follows from Lemma 
\ref{lem:fin}). Then we may take for $L$ the pull-back of 
a linear subgroup of $G/F$ with quotient an abelian variety. 
If $L'$ is another linear subgroup of $G$ such that $G/L'$ 
is an abelian variety, then $L + L'$ is linear, as a quotient 
of $L \times L'$. Moreover, the natural map 
$q : G/L \to G/(L + L')$ is the quotient by $(L + L')/L$, 
a linear subgroup of the abelian variety $G/L$. 
It follows that $(L+L')/L$ is finite; thus, $q$ is an isogeny and 
$L'/L \cap L'$ is finite. Likewise, $q' : G/L' \to G/(L + L')$ 
is an isogeny and $L/L \cap L'$ is finite; this completes 
the proof of (i). 

Finally, the assertion (iii) is a consequence of 
\cite[Thm.~4.3.4]{Brion-II}. 
\end{proof}

The linear algebraic groups may be described 
as follows (see \cite[Thm.~IV.3.1.1]{DG}):

\begin{theorem}\label{thm:lin}
Let $G$ be a linear algebraic group. Then $G$ has a largest
subgroup of multiplicative type, $M$; moreover, $G/M$ is unipotent.
If $k$ is perfect, then $G = M \times U$, where $U$ denotes the
largest unipotent subgroup of $G$. 
\end{theorem}

Also, note the following orthogonality relations:

\begin{proposition}\label{prop:hom}

\begin{enumerate}

\item[{\rm (i)}] Let $M$ be a group of multiplicative type, 
and $U$ a unipotent group. Then 
$\Hom_\cC(M,U) = 0 = \Hom_\cC(U,M)$.

\item[{\rm (ii)}] Let $L$ be a linear algebraic group,
and $A$ an abelian variety. Then $\Hom_\cC(A,L) = 0$, and
every morphism $L \to A$ has finite image. Moreover,
$\Hom_\cC(L,A)$ is $n$-torsion for some positive integer $n$.

\end{enumerate}

\end{proposition}

\begin{proof}
(i) This follows from \cite[Cor.~IV.2.2.4]{DG}.

(ii) The image of a morphism $A \to L$ is proper, smooth,
connected and affine, hence zero. Likewise, the image of 
a morphism $L \to A$ is affine and proper, hence finite.

To show the final assertion, we may replace $k$ with any field 
extension, and hence assume that $k$ is perfect. Then the 
reduced neutral component $L^0_\red$ is a smooth connected subgroup 
of $L$, the quotient $L/L^0_\red$ is finite, and 
$\Hom_\cC(L^0_\red,A) = 0$ by the above argument. Thus, 
$\Hom_\cC(L,A) = \Hom_\cC(L/L^0_\red, A)$ and this group is $n$-torsion,
where $n$ denotes the order of the finite group scheme $L/L^0_\red$ 
(indeed, $L/L^0_\red$ is $n$-torsion in view of 
\cite[VIIA, Prop.~8.5]{SGA3}).
\end{proof}

Next, we obtain a key preliminary result. To state it, recall
that a unipotent group $G$ is \emph{split} if it admits a finite
increasing sequence of subgroups 
$0 = G_0 \subset G_1 \subset \cdots \subset G_n = G$ such that
$G_i/G_{i-1} \cong \bG_a$ for $i = 1, \ldots, n$.

\begin{proposition}\label{prop:sc}
Let $G$ be an algebraic group. 

\begin{enumerate}

\item[{\rm (i)}] 
There exists a finite subgroup $F \subset G$ such that 
$G/F$ is smooth and connected.  

\item[{\rm (ii)}] 
If $G$ is unipotent, then we may choose $F$ such that $G/F$ is split.

\end{enumerate}

\end{proposition}

\begin{proof}
(i) By Lemma \ref{lem:fin}, we have $G = G^0 + F$ for some
finite subgroup $F \subset G$. Thus, $G/F \cong G^0/F \cap G^0$ 
is connected; this completes the proof when $\car(k)=0$. 

When $\car(k) = p > 0$, we may assume $G$ connected by the above step. 
Consider the relative Frobenius morphism $F_{G/k}: G \to G^{(p)}$ 
and its iterates $F^n_{G/k} : G \to G^{(p^n)}$, where $n \geq 1$. 
Then $\Ker(F^n_{G/k})$ is finite for all $n$; moreover, 
$G/\Ker(F^n_{G/k})$ is smooth for $n \gg 0$ (see 
\cite[VIIA, Prop.~8.3]{SGA3}), and still connected. 

(ii) We argue by induction on the dimension of $G$. 
The statement is obvious if $\dim(G) = 0$. In the
case where $\dim(G) = 1$, we may assume that $G$ is smooth 
and connected in view of Lemma \ref{lem:fin} again; then 
$G$ is a $k$-form of $\bG_a$. 
By \cite[Thm.~2.1]{Russell}, there exists an exact sequence
\[ 0 \longrightarrow G \longrightarrow \bG_a^2 
\stackrel{f}{\longrightarrow} \bG_a \longrightarrow 0, \]
where $f \in \cO(\bG_a^2) \cong k[x,y]$ satisfies
$f(x,y) = y^{p^n} - a_0 \, x - a_1 \, x^p - \cdots - a_m \, x^{p^m}$
for some integers $m,n \geq 0$ and some $a_0,\ldots,a_m \in k$
with $a_0 \neq 0$. Thus, the projection 
\[ p_1 : G \longrightarrow \bG_a, \quad (x,y) \longmapsto x \]
lies in an exact sequence
\[ 0 \longrightarrow \alpha_{p^n} \longrightarrow G
\stackrel{p_1}{\longrightarrow} \bG_a \longrightarrow 0, \]
where $\alpha_{p^n}$ denotes the kernel of the endomorphism
$x \mapsto x^{p^n}$ of $\bG_a$. This yields the assertion 
in this case.

If $\dim(G) \geq 2$, then we may choose a subgroup 
$G_1 \subset G$ such that $0 < \dim(G_1) < \dim(G)$
(as follows from \cite[Prop.~IV.2.2.5]{DG}). 
By the induction assumption for $G/G_1$, there exists 
a subgroup $G_2 \subset G$ such that $G_1 \subset G_2$, 
$G_2/G_1$ is finite, and $G/G_2$ is split. Next, the induction 
assumption for $G_2$ yields a finite subgroup $F \subset G_2$ 
such that $G_2/F$ is split. Then $G/F$ is split as well.
\end{proof}

\begin{remark}\label{rem:sc}
By Proposition \ref{prop:sc}, every algebraic group $G$ admits 
an isogeny $u : G \to H$, where $H$ is smooth and connected. 
If $k$ is perfect, then there also exists an isogeny 
$v: K \to G$, where $K$ is smooth and connected: just
take $v$ to be the inclusion of the reduced neutral component 
$G^0_\red$. But this fails over any imperfect field $k$.
Indeed, if such an isogeny $v$ exists, then its image must be 
$G^0_\red$. On the other hand, by \cite[VIA, Ex.~1.3.2]{SGA3}, 
there exists a connected algebraic group $G$ such that $G_\red$ 
is not a subgroup.
\end{remark}

By combining Lemma \ref{lem:fin}, Theorems \ref{thm:che} and
\ref{thm:lin}, and Proposition \ref{prop:sc}, we obtain readily:

\begin{proposition}\label{prop:el}
Every algebraic group $G$ admits a finite increasing sequence
of subgroups
$0 = G_0 \subset G_1 \subset \cdots \subset G_n = G$
such that each $G_i/G_{i-1}$, $i = 1, \ldots, n$, is finite or 
isomorphic to $\bG_a$, a simple torus, or a simple abelian variety.
Moreover, $G$ is linear if and only if no abelian variety occurs.
\end{proposition}

\subsection{Characteristic zero}
\label{subsec:cz}

In this subsection, we assume that $\car(k) = 0$. Recall 
that every unipotent group is isomorphic to the additive group 
of its Lie algebra via the exponential map; this yields an 
equivalence between the category $\cU$ of unipotent groups and 
the category of finite-dimensional $k$-vector spaces (see 
\cite[Prop.~IV.2.4.2]{DG}). In particular, every unipotent
group is connected. 

Next, consider a connected algebraic group $G$. By Theorem 
\ref{thm:che}, there is a unique exact sequence 
$0 \to L \to G \to A \to 0$,
where $A$ is an abelian variety, and $L$ is connected and linear. 
Moreover, in view of Theorem \ref{thm:lin}, we have
$L = T \times U$, where $T$ is a torus and $U$ is unipotent.

We now extend the latter structure results to possibly 
non-connected groups:

\begin{theorem}\label{thm:zero}

\begin{enumerate}

\item[{\rm (i)}] 
Every algebraic group $G$ lies in an exact sequence
\[ 0 \longrightarrow M \times U \longrightarrow G \longrightarrow A
\longrightarrow 0, \]
where $M$ is of multiplicative type, $U$ is unipotent,
and $A$ is an abelian variety. Moreover, $U$ is the largest unipotent
subgroup of $G$: the unipotent radical, $R_u(G)$. Also, $M$ is unique 
up to commensurability in $G$, and $A$ is unique up to isogeny.

\item[{\rm (ii)}] 
The formation of the unipotent radical commutes with base change
under field extensions, and yields an exact functor 
\[ R_u : \cC \longrightarrow \cU, \] 
right adjoint to the inclusion $\cU \to \cC$.

\item[{\rm (iii)}] 
The projective objects of $\cC$ are exactly the unipotent groups.

\end{enumerate}

\end{theorem}

\begin{proof}
(i) Theorem \ref{thm:che} yields an exact sequence
\[ 0 \longrightarrow L \longrightarrow G \longrightarrow A
\longrightarrow 0, \]
where $L$ is linear (possibly non-connected), and $A$ is an
abelian variety. By Theorem \ref{thm:lin}, we have $L = M \times U$, 
where $M$ is of multiplicative type and $U$ is unipotent.

Since $M$ and $A$ have no non-trivial unipotent subgroups, we have
$U = R_u(G)$. Given another exact sequence
\[ 0 \longrightarrow M' \times U \longrightarrow G \longrightarrow A'
\longrightarrow 0 \]
satisfying the same assumptions, the image of $M'$ in 
$A \cong G/(M \times U)$ is finite by Proposition \ref{prop:hom}. 
In other words, the quotient $M'/M' \cap (M \times U)$ is finite. 
Likewise, $M/M \cap (M' \times U)$ is finite as well. Since 
$(M \times U) \cap M' = M \cap M' = M \cap (M' \times U)$,
we see that $M,M'$ are commensurable in $G$. Then 
$A = G/(M \times U)$ and $A' = G/(M' \times U)$ are both quotients
of $G/(M \cap M') \times U$ by finite subgroups, and hence are isogenous. 

(ii) In view of (i), $G/R_u(G)$ is an extension of an abelian variety
by a group of multiplicative type. Since these two classes of
algebraic groups are stable under base change by any field 
extension $k'$ of $k$, it follows that $(G/R_u(G))_{k'}$ has zero
unipotent radical. Thus, $R_u(G)_{k'} = R_u(G_{k'})$. 

Next, note that every morphism $f: G \to H$ sends $R_u(G)$ 
to $R_u(H)$. Consider an exact sequence 
\[ 0 \longrightarrow G_1 \stackrel{f}{\longrightarrow} G_2 
\stackrel{g}{\longrightarrow} G_3 \longrightarrow 0 \] 
and the induced complex
\[ 0 \longrightarrow R_u(G_1) \stackrel{R_u(f)}{\longrightarrow}
R_u(G_2) \stackrel{R_u(g)}{\longrightarrow} R_u(G_3) 
\longrightarrow 0. \]
Clearly, $R_u(f)$ is a monomorphism. Also, we have
$\Ker(R_u(g)) = R_u(G_2) \cap \Ker(g) = R_u(G_2) \cap \Im(f) 
= \Im(R_u(f))$. 
We now show that $R_u(g)$ is an epimorphism. For this, we may
replace $G_2$ with $g^{-1}(R_u(G_3))$, and hence assume that $G_3$
is unipotent. Next, we may replace $G_2$ with $G_2/R_u(G_2)$,
and hence assume (in view of (i) again) that $G_2$ is an extension 
of an abelian variety by a group of multiplicative type. Then 
$\Hom(G_2,G_3) = 0$ by Proposition \ref{prop:hom}; this completes 
the proof of the exactness assertion.

The assertion about adjointness follows from the fact that every 
morphism $U \to H$, where $U$ is unipotent and $H$ arbitrary, 
factors through a unique morphism $U \to R_u(H)$.

(iii) Consider an epimorphism $\varphi : G \to H$, a unipotent group 
$U$, and a morphism $\psi: U \to H$. Then $\psi$ factors through
$R_u(H)$. Also, by (ii), $\varphi$ restricts to an epimorphism
$R_u(G) \to R_u(H)$, which admits a section as unipotent groups
are just vector spaces. Thus, $\psi$ lifts to a morphism $U \to G$.
This shows that $U$ is projective in $\cC$.

Conversely, let $G$ be a projective object in $\cC$. We claim
that the (abstract) group $\Hom_{\cC}(G,H)$ is divisible for any
divisible algebraic group $H$. Indeed, the exact sequence
\[ 0 \longrightarrow H[n] \longrightarrow H
\stackrel{n_H}{\longrightarrow} H \longrightarrow 0 \]
yields an exact sequence 
\[ 0 \longrightarrow \Hom_\cC(G,H[n]) \longrightarrow \Hom_\cC(G,H)  
\stackrel{\times n}{\longrightarrow} \Hom_\cC(G,H) \longrightarrow 0, \]
for any positive integer $n$.

Next, the exact sequence $0 \to L \to G \to A \to 0$ yields 
an exact sequence
\[ 0 \longrightarrow \End_\cC(A) \longrightarrow \Hom_\cC(G,A)  
\longrightarrow \Hom_\cC(L,A), \]
where the abelian group $\End_\cC(A)$ is free of finite rank
(see \cite[Thm.~12.5]{Milne86}),
and $\Hom_\cC(L,A)$ is killed by some positive integer
(Proposition \ref{prop:hom}). On the other hand, $\Hom_{\cC}(G,A)$
is divisible by the above claim.
It follows that $\End_\cC(A)$ is zero, and hence so is $A$.
Thus, $G$ is linear, and hence $G = M \times U$ as above.
Since $U$ is projective, so is $M$. Choose a torus $T$ containing
$M$; then again, the group $\Hom_\cC(M,T)$ is finitely generated and
divisible, hence zero. Thus, $T = 0$ and $G = U$.
\end{proof}

\begin{remark}\label{rem:zero}
With the notation of the above theorem, we have a natural map
\[ G \longrightarrow G/M \times_A G/U, \] 
which is a morphism of $M \times U$-torsors over $A$, and hence 
an isomorphism. Moreover, $G/M$ is an extension of an abelian 
variety by a unipotent group; such `vector extensions' will be
studied in detail in \S \ref{subsec:veav}. Also, $G/U$ is an
extension of an abelian variety by a group of multiplicative
type, and hence of a semi-abelian variety by a finite group.
The semi-abelian varieties will be considered in \S \ref{subsec:sav}.
\end{remark}

\subsection{Positive characteristics}
\label{subsec:pc}

In this subsection, we assume that $\car(k) = p > 0$. Then 
the assertions of Theorem \ref{thm:zero} are no longer valid.
For example, the formation of the unipotent radical (the largest
smooth connected unipotent subgroup) is not exact, and does not
commute with arbitrary field extensions either (see Remark 
\ref{rem:pos} (i) for details). Also, $\cC$ has no non-zero projective
objects, as will be shown in Corollary \ref{cor:proj}. Yet Theorem
\ref{thm:zero} has a useful analogue, in which the unipotent radical
is replaced by the largest unipotent quotient:

\begin{theorem}\label{thm:pos}
Let $G$ be an algebraic group.

\begin{enumerate}

\item[{\rm (i)}] 
$G$ has a smallest subgroup $H$ such that $U := G/H$ is unipotent.
Moreover, $H$ is an extension of an abelian variety $A$ by a group 
of multiplicative type $M$. Also, $M$ is unique up to 
commensurability in $G$, and $A$ is unique up to isogeny.

\item[{\rm (ii)}] Every morphism $H \to U$ is zero; every
morphism $U \to H$ has finite image.

\item[{\rm (iii)}] The formation of $U$ commutes with base
change under field extensions, and yields a functor
\[ U: \cC \longrightarrow \cU, \]
which is left adjoint to the inclusion of $\cU$ in $\cC$.
Moreover, every exact sequence in $\cC$
\[ 0 \longrightarrow G_1 
\stackrel{f}{\longrightarrow} G_2 
\stackrel{g}{\longrightarrow} G_3
\longrightarrow 0 \]
yields a right exact sequence
\[ 0 \longrightarrow U(G_1) 
\stackrel{U(f)}{\longrightarrow} U(G_2) 
\stackrel{U(g)}{\longrightarrow} U(G_3)
\longrightarrow 0, \]
where $\Ker(U(f))$ is finite.

\item[{\rm (iv)}] There exists a subgroup $V \subset G$ such that 
$G = H + V$ and $H \cap V$ is a finite $p$-group.

\end{enumerate}

\end{theorem}

\begin{proof}
(i) Since the underlying topological space of $G$ is noetherian,
we may choose a subgroup $H \subset G$ such that $G/H$ is
unipotent, and $H$ is minimal for this property. Let $H' \subset G$
be another subgroup such that $G/H'$ is unipotent. Then so is
$G/H \cap H'$ in view of the exact sequence
\[ 0 \longrightarrow  G/H \cap H' \longrightarrow G/H \times G/H'. \]
By minimality of $H$, it follows that $H \cap H' = H$, i.e.,
$H \subset H'$. Thus, $H$ is the smallest subgroup with unipotent
quotient.

Since the class of unipotent groups is stable under extensions, 
every unipotent quotient of $H$ is zero. Also, by the 
affinization theorem (see \cite[Thm.~1, Prop.~5.5.1]{Brion-II}), 
$H$ is an extension of a linear algebraic group $L$ by a semi-abelian 
variety $S$. Then every unipotent quotient of $L$ is zero,
and hence $L$ must be of multiplicative type in view of 
Theorem \ref{thm:lin}. By \cite[Cor.~IV.1.3.9]{DG},
the reduced neutral component $L^0_{\red}$ is its maximal torus,
$T$; the quotient $L/T$ is a finite group of multiplicative 
type. Denote by $S'$ the preimage of $T$ in $H$; then $S'$
is a semi-abelian variety (extension of $T$ by $S$) and we have 
an exact sequence 
\[ 0 \longrightarrow S' \longrightarrow H \longrightarrow 
L/T \longrightarrow 0. \]
By Lemma \ref{lem:fin}, there exists a finite subgroup 
$F \subset H$ such that $H = S' + F$; equivalently, 
the quotient map $H \to L/T$ restricts to an epimorphism 
$F \to L/T$. Also, by Theorem \ref{thm:lin} again, 
$F$ has a largest subgroup of multiplicative type, $M_F$, 
and the quotient $F/M_F$ is unipotent. Since $L/T$ is
of multiplicative type, it follows that the composition 
$M_F \to F \to L/T$ is an epimorphism as well.
Thus, we may replace $F$ with $M_F$, and assume that 
$F$ is of multiplicative type. Let $T'$ be the maximal 
torus of the semi-abelian variety $S'$, and $M := T' + F$. 
Then $M$ is of multiplicative type; moreover, $H/M$ is 
a quotient of $S'/T'$, and hence is an abelian variety. 
The uniqueness assertions may be checked as in the proof of 
Theorem \ref{thm:zero}.

(ii) This follows readily from Proposition \ref{prop:hom}.

(iii) The assertion on base change under field extensions
follows from the stability of the classes of unipotent groups,
abelian varieties, and groups of multiplicative type, under
such base changes. The adjointness assertion may be checked
as in the proof of Theorem \ref{thm:zero} (ii).

Next, consider an exact sequence as in the statement. Clearly,
$U(g)$ is an epimorphism. Also, 
$\Ker(U(g))/\Im(U(f)) \cong g^{-1}(H_3)/(u(G_1) + H_2)$,
where $H_i$ denotes the kernel of the quotient map 
$G_i \to U(G_i)$. Thus, $\Ker(U(g))/\Im(U(f))$ is a quotient
of $g^{-1}(H_3)/u(G_1) \cong H_3$. As $\Ker(U(g))/\Im(U(f))$
is unipotent, it is trivial by (ii). Finally,
$\Ker(U(f)) \cong (G_1 \cap u^{-1}(H_2))/H_1$ is isomorphic
to a subgroup of $H_2/u(H_1)$. Moreover, $H_2/u(H_1)$ is an
extension of an abelian variety by a group of multiplicative
type. Since $\Ker(U(f))$ is unipotent, it is finite by (ii) 
again. 

(iv) Consider the subgroups $\Ker(p^n_G) \subset G$, where $n$ 
is a positive integer. Since they form a decreasing sequence,
there exists a positive integer $m$ such that 
$\Ker(p^n_G) = \Ker(p^m_G)$
for all $n \geq m$. Let $\overline{G} := G/\Ker(p^m_G)$,
then $\Ker(p_{\overline{G}}) = \Ker(p^{m + 1}_G)/\Ker(p^m_G)$
is zero, and hence $p_{\overline{G}}$ is an isogeny. Next, 
let $H \subset G$ be as in (i) and put 
$\overline{H} := H/\Ker(p^m_H)$,
$\overline{U} = \overline{G}/\overline{H}$. Then $\overline{U}$ 
is unipotent (as a quotient of $U$) and $p_{\overline{U}}$ has 
finite cokernel (since this holds for $p_{\overline{G}}$). Thus, 
$\overline{U}$ is a finite $p$-group. By Lemma \ref{lem:fin}, 
there exists a finite $p$-subgroup $F \subset G$ such that 
$\overline{G} = \overline{H} + \overline{F}$
with an obvious notation. Thus, $G = H + V$, where 
$V := \Ker(p^m_G) + F$. Also,  
$H \cap \Ker(p^m_G) = \Ker(p^m_H)$ is finite, since $H$
is an extension of $A$ by $M$, and $\Ker(p^m_A)$ and 
$\Ker(p^m_M)$ are finite. As $F$ is finite, it follows that 
$H \cap V$ is finite as well. Moreover, $V$ is a $p$-group,
since so are $F$ and $\Ker(p^m_G)$; we conclude that $H \cap V$ 
is a finite $p$-group.
\end{proof}

\begin{remarks}\label{rem:pos}
(i) The formation of the unipotent radical does not commute with
purely inseparable field extensions, in view of 
\cite[XVII.C.5]{SGA3}. This formation is not exact either, as seen
e.g. from the exact sequence
\[ 0 \longrightarrow \alpha_p \longrightarrow \bG_a
\stackrel{F}{\longrightarrow} \bG_a
\longrightarrow 0, \]
where $F$ denotes the relative Frobenius endomorphism. 

When $k$ is perfect, one may show that every exact sequence
$0 \to G_1 \to G_2 \to G_3 \to 0$ in $\cC$ yields a complex
$0 \to R_u(G_1) \to R_u(G_2) \to R_u(G_3) \to 0$ with finite
homology groups. But this fails when $k$ is imperfect; more
specifically, choose a finite purely inseparable field extension
$K$ of $k$ of degree $p$, and consider $G := \R_{K/k}(\bG_{m,K})$,
where $\R_{K/k}$ denotes the Weil restriction. By 
\cite[Prop.~A.5.11]{CGP}, $G$ is smooth, connected, and lies in 
an exact sequence
\[ 0 \longrightarrow \bG_m \longrightarrow G
\longrightarrow U\longrightarrow 0, \]
where $U$ is unipotent of dimension $p - 1$. Moreover, every morphism 
from a smooth connected unipotent group to $G$ is constant, as follows 
from the adjointness property of the Weil restriction (see
\cite[(A.5.1)]{CGP}). In other terms, $R_u(G) = 0$.

(ii) The functor $U$ of Theorem \ref{thm:pos} is not left exact, 
as seen from the exact sequence 
\[ 0 \longrightarrow \alpha_p \longrightarrow E 
\stackrel{F}{\longrightarrow} E^{(p)}
\longrightarrow 0, \]
where $E$ denotes a supersingular elliptic curve, and $F$ 
its relative Frobenius morphism. Also, note that the torsion
subgroups $E[p^n]$, where $n \geq 1$, form a strictly increasing 
sequence of infinitesimal unipotent groups; in particular, 
$E$ has no largest connected unipotent subgroup.
\end{remarks}

\begin{corollary}\label{cor:pos}
Let $G$ be an algebraic group.

\begin{enumerate}

\item[{\rm (i)}]
There exists a finite subgroup $F \subset G$ such that
$G/F \cong S \times U$, where $S$ is a semi-abelian variety,
and $U$ a split unipotent group. Moreover, $S$ and $U$ are 
unique up to isogeny.

\item[{\rm (ii)}] If $k$ is locally finite, then we may choose 
$F$ so that $S \cong T \times A$, where $T$ is a torus, and 
$A$ an abelian variety. Moreover, $T$ and $A$ are unique up
to isogeny.

\end{enumerate}

\end{corollary}

\begin{proof}
(i) With the notation of Theorem \ref{thm:pos}, we have isomorphisms
\[ G/H \cap V \cong G/V \times G/H \cong (H/H \cap V) \times U. \]
Also, $H/H \cap V$ is an extension of an abelian variety,
$H/(H \cap V) + M$, by a group of multiplicative type,
$M/M \cap V$. Moreover, $U$ is an extension of a split
unipotent group by a finite group (Proposition \ref{prop:sc}).
Thus, we may assume that $G = H$. Then $G^0_\red$ is a semi-abelian
variety, as follows from \cite[Lem.~5.6.1]{Brion-II}. Since
$G/G^0_\red$ is finite, applying Lemma \ref{lem:fin} yields that
$G$ is an extension of a semi-abelian variety by a finite group. 

(ii) By \cite[Cor.~5.5.5]{Brion-II}, there exists an abelian
subvariety $A \subset S$ such that $S = T + A$, where $T \subset S$
denotes the maximal torus. Then $T \cap A$ is finite, and 
$S/T \cap A \cong T/T \cap A \times A/T \cap A$. 

This completes the proof of the existence assertions in (i) and (ii). 
The uniqueness up to isogeny follows from Proposition \ref{prop:hom}.
\end{proof}

\section{The isogeny category of algebraic groups}
\label{sec:ic}

\subsection{Definition and first properties}
\label{subsec:dfp}

Recall that $\cC$ denotes the category of commutative algebraic
groups, and $\cF$ the full subcategory of finite groups. Since $\cF$
is stable under taking subobjects, quotients and extensions, we may 
form the quotient category $\cC/\cF$; it has the same objects as 
$\cC$, and its morphisms are defined by  
\[ \Hom_{\cC/\cF}(G,H) = \lim_{\to} \Hom_\cC(G',H/H'), \]
where the direct limit is taken over all subgroups $G' \subset G$
such that $G/G'$ is finite, and all finite subgroups $H' \subset H$.
The category $\cC/\cF$ is abelian, and comes with an exact functor
\[ Q : \cC \longrightarrow \cC/\cF, \] 
which is the identity on objects and the natural map 
\[ \Hom_\cC(G,H) \longrightarrow \lim_{\to} \Hom_\cC(G',H/H'),
\quad f \longmapsto \uf \]
on morphisms. The quotient functor $Q$ satisfies the following 
universal property:
given an exact functor $R : \cC \to \cD$, where $\cD$ is an abelian
category, such that $R(F) = 0$ for any finite group $F$, there exists
a unique exact functor $S : \cC/\cF \to \cD$ such that $R = S \circ Q$
(see \cite[Cor.~III.1.2, Cor.~III.1.3]{Gabriel} for these results).

Alternatively, $\cC/\cF$ may be viewed as the localization of $\cC$
at the multiplicative system of isogenies (see  \cite[\S I.2]{GZ} or
\cite[\S 4.26]{SP} for localization of categories); this is easily 
checked by arguing as in the proof of \cite[Lem.~12.9.6]{SP}.

We now show that $\cC/\cF$ is equivalent to a category with 
somewhat simpler objects and morphisms:

\begin{lemma}\label{lem:equi}
Let $\ucC$ be the full subcategory of $\cC/\cF$ with objects 
the smooth connected algebraic groups.

\begin{enumerate}

\item[{\rm (i)}] 
The inclusion of $\ucC$ in $\cC/\cF$ is an equivalence of
categories.

\item[{\rm (ii)}]
$\Hom_{\ucC}(G,H) = \lim \Hom_\cC(G,H/H')$,
where the direct limit is taken over all finite subgroups 
$H' \subset H$.

\item[{\rm (iii)}]  
Let $\uf \in \Hom_{\ucC}(G,H)$ be represented by a morphism
$f : G \to H/H'$ in $\cC$. Then $\uf$ is zero (resp.~a monomorphism, 
an epimorphism, an isomorphism) if and only if $f$ is zero 
(resp.~has a finite kernel, is an epimorphism, is an isogeny).

\end{enumerate}

\end{lemma}

\begin{proof}
(i) This follows from Proposition \ref{prop:sc}.

(ii) This follows from Lemma \ref{lem:epi}.

(iii) By \cite[Lem.~III.1.2]{Gabriel}, $\uf$ is zero  
(resp.~a monomorphism, an epimorphism) if and only if 
$\Im(f)$ (resp.~$\Ker(f)$, $\Coker(f)$) is finite.
By Lemma \ref{lem:epi} again, the finiteness of $\Im(f)$ 
is equivalent to $f = 0$, and the finiteness of $\Coker(f)$ 
is equivalent to $f$ being an epimorphism. As a consequence, 
$\uf$ is an isomorphism if and only if $f$ is an isogeny.
\end{proof}

The abelian category $\ucC$ will be called the
\emph{isogeny category of (commutative) algebraic groups}. 
Every exact functor $R : \cC \to \cD$, where $\cD$ is an abelian
category and $R(f)$ is an isomorphism for any isogeny $f$, 
factors uniquely through $\ucC$ (indeed, $R$ must send 
any finite group to zero).

\medskip

We may now prove the assertion (i) of the main theorem:

\begin{proposition}\label{prop:simple}

\begin{enumerate}

\item[{\rm (i)}] The category $\ucC$ is noetherian and artinian. 

\item[{\rm (ii)}] The simple objects of $\ucC$ are exactly 
$\bG_a$, the simple tori, and the simple abelian varieties.

\end{enumerate}

\end{proposition}

\begin{proof}
(i) Let $G$ be a smooth connected algebraic group, and $(G_n)_{n \geq 0}$
an increasing sequence of subobjects of $G$ in $\ucC$, i.e.,
each $G_n$ is smooth, connected, and equipped with a $\cC$-morphism
\[ \varphi_n : G_n \longrightarrow G/G'_n, \]
where $\Ker(\varphi_n)$ and $G'_n$ are finite; moreover, we have 
$\cC$-morphisms
\[ \psi_n : G_n \longrightarrow G_{n+1}/G''_{n+1}, \]
where $\Ker(\psi_n)$ and $G''_{n+1}$ are finite. Thus, 
$\dim(G_n) \leq \dim(G_{n+1}) \leq \dim(G)$. It follows that 
$\dim(G_n) = \dim(G_{n+1})$ for $n \gg 0$, and hence $\psi_n$ 
is an isogeny. So $G_n \cong G_{n+1}$ in $\ucC$ for $n \gg 0$. 
This shows that $\ucC$ is noetherian. One may check likewise that 
$\ucC$ is artinian. 

(ii) This follows from Proposition \ref{prop:el}.
\end{proof}

Next, we relate the short exact sequences in $\cC$ with those
in $\ucC$:

\begin{lemma}\label{lem:exa}
Consider a short exact sequence in $\cC$,
\[ \xi : \quad 0 \longrightarrow G_1 
\stackrel{u}{\longrightarrow} G_2 
\stackrel{v}{\longrightarrow} G_3 
\longrightarrow 0, \]
where $G_1,G_2,G_3$ are smooth and connected.
Then $\xi$ splits in $\ucC$ if and only if the push-out
$f_* \xi$ splits in $\cC$ for some epimorphism
with finite kernel $f : G_1 \to H$.
\end{lemma}

\begin{proof}
Recall that $\xi$ splits in $\ucC$ if and only if
there exists a $\ucC$-morphism $\ug: G_2 \to G_1$
such that $\ug \circ u = \id$ in $\ucC$. Equivalently,
there exists a finite subgroup $G'_1 \subset G_1$ and 
a $\cC$-morphism $g: G_2 \to G_1/G'_1$ such that 
$g \circ u$ is the quotient map $f_1 : G_1 \to G_1/G'_1$.

If such a pair $(G'_1,g)$ exists, then $g$ factors through
a morphism $G_2/u(G'_1) \to G_1/G'_1$, which splits the bottom
exact sequence in the push-out diagram
\[ \CD
0 @>>> G_1 @>{u}>> G_2 @>{v}>> G_3 @>>> 0 \\
& & @V{f_1}VV @V{f_2}VV @V{\id}VV \\
0 @>>> G_1/G'_1 @>{u'}>> G_2/u(G'_1) @>{v'}>> G_3 @>>> 0. \\
\endCD \]
Replacing $G'_1$ by a larger finite subgroup, we may assume that 
$G_1/G'_1$ is smooth and connected (Lemma \ref{lem:fin}).

Conversely, a splitting of the bottom exact sequence in the above
diagram is given by a $\cC$-morphism $g' : G_2/u(G'_1) \to G_1/G'_1$ 
such that $g' \circ u' = \id$ in $\cC$. Let $g : G_2 \to G_1/G'_1$ 
denote the composition
$G_2 \stackrel{f_2}{\longrightarrow} G_2/u(G'_1) 
\stackrel{g'}{\longrightarrow} G_1/G'_1$.
Then 
$g \circ u = g' \circ f_2 \circ u = g' \circ u' \circ f_1 = f_1$
as desired.
\end{proof}

We may now construct non-split exact sequences in $\ucC$,
thereby showing that $\hd(\ucC) \geq 1$:

\begin{examples}\label{ex:nonsplit}
(i) Consider an exact sequence
\[ \xi : \quad 0 \longrightarrow \bG_a \longrightarrow G
\longrightarrow A \longrightarrow 0, \]
where $A$ is an abelian variety. Then $\xi$, viewed as an
extension of $A$ by $\bG_a$ in $\cC$, is classified by an element
$\eta \in H^1(A,\cO_A)$ (see \cite{Rosenlicht-II} or 
\cite[\S 1.9]{MM}). 

If $\car(k) = 0$, then every epimorphism with finite kernel 
$f : \bG_a \to H$ may be identified with the multiplication 
by some $t \in k^*$, viewed as an endomorphism of $\bG_a$; 
then the push-out $f_* \xi$ is classified by $t \eta$. 
In view of Lemma \ref{lem:exa}, it follows that $\xi$ 
is non-split in $\ucC$ whenever $\eta \neq 0$.

In contrast, if $\car(k) = p > 0$, then $\xi$ splits in
$\ucC$. Indeed, the multiplication map $p_A$ yields an isomorphism 
in $\ucC$, and $p_{\bG_a} = 0$ whereas 
$(p_{\bG_a})_* \xi = p_A^* \xi$.

(ii) Assume that $\car(k) = p > 0$ and consider the
algebraic group $W_2$ of Witt vectors of length $2$.
This group comes with an exact sequence
\[ \xi : \quad 0 \longrightarrow \bG_a \longrightarrow W_2
\longrightarrow \bG_a \longrightarrow 0, \]
see e.g. \cite[\S V.1.1.6]{DG}. Every epimorphism with finite
kernel $f : \bG_a \to H$ may be identified
with a non-zero endomorphism of $\bG_a$. In view of
\cite[Cor.~V.1.5.2]{DG}, it follows that the push-forward 
$f_*\xi$ is non-split. Thus, $\xi$ does not split in $\ucC$.
\end{examples}

\begin{proposition}\label{prop:exa}
Consider an exact sequence 
\[ 0 \longrightarrow G_1 
\stackrel{\uu_1}{\longrightarrow} G_2
\stackrel{\uu_2}{\longrightarrow}  \cdots
 \stackrel{\uu_{n-1}}{\longrightarrow}
G_n \longrightarrow 0 \] 
in $\ucC$. Then there exists an exact sequence
\[ 0 \longrightarrow H_1 
\stackrel{v_1}{\longrightarrow} H_2 
\stackrel{v_2}{\longrightarrow} \cdots 
\stackrel{v_{n-1}}{\longrightarrow}
H_n \longrightarrow 0 \] 
in $\cC$, and epimorphisms with finite kernels 
$f_i : G_i \to H_i$ ($i = 1, 2, \ldots, n$), such that the diagram
\[ \CD
0 @>>> G_1 @>{\uu_1}>> G_2 @>{\uu_2}>> & \cdots & @>{\uu_{n-1}}>> G_n @>>> 0 \\
& & @V{\uf_1}VV @V{\uf_2}VV  &  & & & @V{\uf_n}VV \\
0 @>>> H_1 @>{\uv_1}>> H_2 @>{\uv_2}>> & \cdots & @>{\uv_{n-1}}>> H_n @>>> 0 \\
\endCD \]
commutes in $\ucC$.
\end{proposition}

\begin{proof}
We argue by induction on the length $n$. If $n = 2$, then
we just have an isomorphism $\uu : G_1 \to G_2$ in $\ucC$. Then $\uu$ 
is represented by an isomorphism $u: G_1/G'_1 \to G_2/G'_2$ in $\cC$,
for some finite groups $G'_1 \subset G_1$ and $G'_2 \subset G_2$.

If $n = 3$, then the $\ucC$-morphism $\uu_2 : G_2 \to G_3$ is represented 
by an epimorphism $u_2: G_2 \to G_3/G'_3$ in $\cC$, where $G'_3$ is 
a finite subgroup of $G_3$. We may thus replace $G_3$ with $G_3/G'_3$, 
and assume that $u_2$ is an epimorphism in $\cC$.

Next, $\uu_1: G_1 \to G_2$ is represented by a morphism 
$u_1: G_1 \to G_2/G'_2$ with finite kernel, where $G'_2$ is a finite 
subgroup of $G_2$. We may thus replace $G_1$ (resp.~$G_2$, $G_3$)
with $G_1/\Ker(u_1)$ (resp.~$G_2/G'_2$, $G_3/u_2(G'_2)$)
and assume that $u_1$ is a monomorphism in $\cC$.
Then $u_2 \circ u_1$ has finite image, and hence is zero
since $G_1$ is smooth and connected.

We now have a complex in $\cC$
\[ 0 \longrightarrow G_1 \stackrel{u_1}\longrightarrow G_2 
\stackrel{u_2}{\longrightarrow} G_3 \longrightarrow 0, \]
where $u_1$ is a monomorphism, $u_2$ an epimorphism, and
$\Ker(u_2)/\Im(u_1)$ is finite. By Lemma \ref{lem:fin},
we may choose a finite subgroup $F \subset \Ker(u_2)$
such that $\Ker(u_2) = \Im(u_1) + F$. This yields
a commutative diagram in $\cC$
\[ \CD
0 @>>> G_1 @>{u_1}>> G_2 @>{u_2}>> G_3 @>>> 0 \\
& & @V{q_1}VV @V{q_2}VV @V{\id}VV \\
0 @>>> G_1/u_1^{-1}(F) @>{v_1}>> G_2/F @>{v_2}>> G_3 @>>> 0, \\
\endCD \]
where $q_1,q_2$ denote the quotient maps.
Clearly, $v_1$ is a monomorphism, and $v_2$ an epimorphism.
Also, $v_2 \circ v_1 = 0$, since $q_1$ is an epimorphism. 
Finally, $q_2$ restricts to an epimorphism $\Ker(u_2) \to \Ker(v_2)$, 
and hence $\Ker(v_2) = \Im(v_1)$. This completes the proof in the
case where $n = 3$.

For an arbitrary length $n \geq 4$, we cut the given exact sequence
into two exact sequences in $\ucC$
\[ 
0 \longrightarrow G_1 \stackrel{\uu_1}{\longrightarrow} G_2 
\longrightarrow K \longrightarrow 0, \]
\[ 0 \longrightarrow K \longrightarrow G_3 
\stackrel{\uu_3}{\longrightarrow} \cdots 
\stackrel{\uu_{n-1}}{\longrightarrow} G_n \longrightarrow 0. \]
By the induction assumption, there exists a commutative diagram 
in $\ucC$
\[ \CD
0 @>>> K @>>> G_3 @>{\uu_3}>> \cdots @>{\uu_{n-1}}>> G_n @>>> 0 \\
& & @V{\uf}VV @V{\uf_3}VV  & &  @V{\uf_n}VV \\
0 @>>> L @>{\uv}>> H_3 @>{\uv_3}>> \cdots @>{\uv_{n-1}}>> H_n @>>> 0, \\
\endCD \]
where $f,f_3,\ldots,f_n$ are epimorphisms with finite kernels,
and the bottom sequence comes from an exact sequence in $\cC$. 
Since $\uf$ is an isomorphism in $\ucC$, we have an exact sequence 
\[ 0 \longrightarrow G_1 \stackrel{\uu_1}{\longrightarrow} G_2 
\longrightarrow L \longrightarrow 0 \]
in $\ucC$, and hence another commutative diagram in $\ucC$
\[ \CD
0 @>>> G_1 @>{\uu_1}>> G_2 @>>> L @>>> 0 \\
& & @V{\uf_1}VV @V{\uf_2}VV @V{\ug}VV \\
0 @>>> H_1 @>{\uv_1}>> H_2 @>>> M @>>> 0, \\
\endCD \]
where again $f_1,f_2,g$ are epimorphisms with finite kernels, 
and the bottom sequence comes from an exact sequence in $\cC$.
Denote by $F$ the kernel of $g : L \to M$; then we have 
a commutative diagram in $\ucC$
\[ \CD
0 @>>> L @>>> G_3 @>{\uu_3}>> G_4 @>{\uu_4}>> 
\cdots @>{\uu_{n-1}}>> G_n @>>> 0 \\
& & @V{\ug}VV @V{\uh}VV  @V{\uf_4}VV & &   @V{\uf_n}VV \\
0 @>>> M @>>> H_3/v(F) @>{\uw}>> H_4 @>{\uv_4}>> 
\cdots @>{\uv_{n-1}}>> H_n @>>> 0, \\
\endCD \]
satisfying similar properties. This yields the desired
commutative diagram
\[ \CD
0 @>>> G_1 @>{\uu_1}>> G_2 @>{\uu_2}>> G_3 @>{\uu_3}>> G_4 
@>{\uu_4}>> \cdots @>{\uu_{n-1}}>> G_n @>>> 0 \\
& & @VVV @VVV @VVV @VVV & & @VVV \\
0 @>>> H_1 @>{\uv_1}>> H_2 @>>> H_3/v(F) @>{\uw}>> H_4 
@>{\uv_4}>> \cdots @>{\uv_{n-1}}>> H_n @>>> 0. \\
\endCD \]

\end{proof}

\subsection{Divisible groups}
\label{subsec:dg}

Given a divisible algebraic group $G$ and a non-zero integer $n$, 
the morphism $n_G : G \to G$ factors through an isomorphism 
$G/G[n] \stackrel{\cong}{\longrightarrow} G$.
We denote the inverse isomorphism by
\[ u_n : G \stackrel{\cong}{\longrightarrow} G/G[n]. \]
By construction, we have a commutative triangle
\[ 
\xymatrix{
G \ar[d]_{n_G} \ar[dr]^{q} \\
G \ar[r]^-{u_n} & G/G[n], \\}
\]
where $q$ denotes the quotient morphism. Since $q$ yields the identity
morphism in $\ucC$, we see that $u_n$ yields the inverse of the 
$\ucC$-automorphism $n_G$ of $G$. As a consequence, $\End_{\ucC}(G)$
is a $\bQ$-algebra.

More generally, we have the following:

\begin{proposition}\label{prop:homdiv}
Let $G,H$ be smooth connected algebraic groups, and assume that 
$H$ is divisible.

\begin{enumerate}

\item[{\rm (i)}]  
Every extension group $\Ext^n_{\ucC}(G,H)$ is a $\bQ$-vector space.

\item[{\rm (ii)}]  
The natural map $Q : \Hom_{\cC}(G,H) \to \Hom_{\ucC}(G,H)$ 
is injective and induces an isomorphism
\[ \gamma : \bQ \otimes_{\bZ} \Hom_{\cC}(G,H) 
\longrightarrow \Hom_{\ucC}(G,H), \quad
\frac{1}{n} \otimes f \longmapsto Q(u_n \otimes f) . \]

\item[{\rm (iii)}]
If $G$ is divisible as well, then the natural map
$Q^1 : \Ext^1_{\cC}(G,H) \to \Ext^1_{\ucC}(G,H)$ 
induces an isomorphism
\[ \gamma^1 : \bQ \otimes_{\bZ} \Ext^1_{\cC}(G,H) 
\longrightarrow \Ext^1_{\ucC}(G,H). \]

\end{enumerate}

\end{proposition}

\begin{proof}
(i) Just note that $\Ext^n_{\ucC}(G,H)$ is a module over 
the $\bQ$-algebra $\End_{\ucC}(H)$.

(ii) Let $f \in \Hom_{\cC}(G,H)$, and $n$ a positive integer.
If $\gamma(\frac{1}{n} \otimes f) = 0$, then $\gamma(f) = 0$, i.e., 
$\uf = 0$. Thus, $f = 0$ by Lemma \ref{lem:equi}. This shows 
the injectivity of $\gamma$.  

For the surjectivity, consider a $\ucC$-morphism $\uf : G \to H$ 
represented by a $\cC$-morphism $f : G \to H/H'$, where $H'$ is 
a finite subgroup of $H$. Then $H' \subset H[n]$ for some positive 
integer $n$, 
which we may take to be the order of $H'$. Thus, we may assume that 
$H' = H[n]$. Then the $\cC$-morphism 
$\varphi := u_n^{-1} \circ f : G \to H$ satisfies 
$\underline{\varphi} = n \uf$, i.e., 
$\uf = \gamma(\frac{1}{n} \otimes \varphi)$.

(iii) Consider $\eta \in \Ext^1_{\cC}(G,H)$ such that 
$\gamma^1(\frac{1}{n} \otimes \eta) = 0$ for some positive integer 
$n$. Then of course $\gamma^1(\eta) = 0$, i.e., $\eta$ is represented
by an exact sequence in $\cC$ 
\[ 0 \longrightarrow H \stackrel{u}{\longrightarrow} E
\stackrel{v}{\longrightarrow} G \longrightarrow 0, \]
which splits in $\ucC$. By Lemma \ref{lem:exa} and the
divisibility of $H$, it follows that the push-out by $m_H$ 
of the above extension splits in $\cC$ for some $m > 0$. 
But $(m_H)_* \eta = m \eta$ (see e.g. \cite[Lem.~I.3.1]{Oort66}), 
and hence $m \eta = 0$. This shows the injectivity of $\gamma^1$.

For the surjectivity, we adapt the argument of Proposition 
\ref{prop:exa}. Let $\eta \in \Ext^1_{\ucC}(G,H)$ be represented 
by an exact sequence in $\ucC$
\[ 0 \longrightarrow H \stackrel{\uu}{\longrightarrow} E
\stackrel{\uv}{\longrightarrow} G \longrightarrow 0. \]
Since $G$ is divisible, $\uv$ is represented by a $\cC$-morphism
$v: E \to G/G[m]$ for some positive integer $m$. Replacing $\eta$
with its pull-back 
$u_m^* \eta = (m_G^*)^{-1} \eta = \gamma^1(\frac{1}{m} \otimes \eta)$, 
we may thus assume that $\uv$ is represented by a $\cC$-epimorphism 
$v : E \to G$.

Likewise, since $H$ is divisible, $\uu$ is represented
by some $\cC$-morphism $u : H \to E/E[n]$. Then $\eta$
is represented by the exact sequence in $\ucC$ 
\[ 0 \longrightarrow H \stackrel{\uu}{\longrightarrow} E/E[n]
\stackrel{\uv_n}{\longrightarrow} G \longrightarrow 0, \]
where $v_n : E/E[n] \to G/G[n]$ is the $\cC$-epimorphism induced 
by $v$. So we may further assume that $u$ is represented by
a $\cC$-morphism $u : H \to E$. By Lemma \ref{lem:equi},
we then have $v \circ u = 0$; moreover, $\Ker(u)$ and 
$\Ker(v)/\Im(u)$ are finite. In view of Lemma \ref{lem:fin}, 
we have $\Ker(v) = \Im(u) + E'$ for some finite subgroup 
$E' \subset E$. This yields a commutative diagram in $\cC$
\[ \CD
0 @>>> H @>{u}>> E @>{v}>> G @>>> 0 \\
& & @VVV @VVV @V{\id}VV \\
0 @>>> H/u^{-1}(E') @>{u_1}>> E/E' @>{v_1}>> G @>>> 0, \\
\endCD \]
where the bottom sequence is exact, and $u^{-1}(E')$ is finite.

We may thus choose a positive integer $r$ such that 
$u^{-1}(E') \subset H[r]$. Taking the push-out by the quotient map 
$H/u^{-1}(E') \to H/H[r]$ yields a commutative diagram in $\cC$
\[ \CD
0 @>>> H @>{u}>> E @>{v}>> G @>>> 0 \\
& & @VVV @VVV @V{\id}VV \\
0 @>>> H/H[r] @>{u_2}>> E/E'+ u(H[r]) @>{v_2}>> G @>>> 0, \\
\endCD \]
where the bottom sequence is exact again. Thus, 
$r \eta = (r_H)_* \eta$ is represented by an exact sequence in $\cC$.
\end{proof}

\begin{remarks}\label{rem:div}
(i) Given two divisible groups $G,H$, the map 
\[ Q^1 : \Ext^1_{\cC}(G,H) \longrightarrow \Ext^1_{\ucC}(G,H) \] 
is not necessarily injective. Indeed, the group 
$\Ext^1_{\cC}(A,\bG_m)$ has non-zero torsion for any non-zero 
abelian variety $A$ over (say) a separably closed field,
as follows from the Weil-Barsotti isomorphism.

(ii) We may also consider the natural maps
\[ \gamma^n : \bQ \otimes_{\bZ} \Ext^n_{\cC}(G,H) 
\longrightarrow \Ext^n_{\ucC}(G,H) \] 
for $n \geq 2$.
But these maps turn out to be zero for any algebraic groups $G,H$, 
since $\Ext^n_{\ucC}(G,H) = 0$ (Lemma \ref{lem:van}) and 
$\Ext^n_{\cC}(G,H)$ is torsion (Remark \ref{rem:van}).
\end{remarks}

As a first application of Proposition \ref{prop:homdiv}, we obtain:

\begin{proposition}\label{prop:ru}
Assume that $\car(k) = 0$.

\begin{enumerate}

\item[{\rm (i)}] The composition of the inclusion $\cU \to \cC$ 
with the quotient functor $Q : \cC \to \ucC$ identifies $\cU$ 
with a full subcategory of $\ucC$. 

\item[{\rm (ii)}]  
The unipotent radical functor yields an exact functor 
\[ \uR_u : \ucC \longrightarrow \cU, \] 
which is right adjoint to the inclusion. Moreover, $\uR_u$ commutes 
with base change under field extensions.

\item[{\rm (iii)}]  
Every unipotent group is a projective object in $\ucC$.

\end{enumerate}

\end{proposition}

\begin{proof}
(i) Recall that a morphism of unipotent groups is just a linear
map of the associated $k$-vector spaces. In view of Proposition
\ref{prop:homdiv}, it follows that the natural map
$\Hom_\cC(U,V) \to \Hom_{\ucC}(U,V)$ is an isomorphism for
any unipotent groups $U,V$. 

(ii) The functor $R_u : \cC \to \cU$ is exact by Theorem 
\ref{thm:zero}, and sends every finite group to $0$. By the 
universal property of $Q$, there exists a unique exact functor 
$S : \cC/\cF \to \cU$ such that $R_u = S \circ Q$. 
Since $R_u$ commutes with base change under field extensions 
(Theorem \ref{thm:zero} again), so does $S$ by uniqueness. 
Thus, composing $S$ with the inclusion $\ucC \to \cC/\cF$
yields the desired exact functor.

For any unipotent group $U$ and any algebraic group $G$, the 
natural map 
\[ \Hom_\cU(U,R_u(G)) \longrightarrow \Hom_\cC(U,G) \] 
is an isomorphism. By Proposition \ref{prop:homdiv} again, 
the natural map 
\[ Q : \Hom_\cC(U,G) \longrightarrow \Hom_{\ucC}(U,G) \] 
is an isomorphism as well. It follows that $\uR_u$ is right adjoint 
to the inclusion.

(iii) Let $U$ be a unipotent group. Then the functor on $\cC$ 
defined by 
\[ G \longmapsto \Hom_\cU(U,R_u(G)) \]
is exact, since the unipotent radical functor is exact and 
the category $\cU$ is semi-simple. Thus, 
$G \mapsto \Hom_{\ucC}(U,G)$ is exact as well; this yields 
the assertion.
\end{proof}

\subsection{Field extensions}
\label{subsec:fe}

Let $k'$ be a field extension of $k$. Then the assignment 
$G \mapsto G_{k'} := G \otimes_k k'$ yields the \emph{base change functor}
\[ \otimes_k k' : \cC_k \longrightarrow \cC_{k'}. \]
Clearly, this functor is exact and faithful; also, note that
$G$ is connected (resp.~smooth, finite, linear, unipotent, a torus,
an abelian variety, a semi-abelian variety) if and only if so is $G_{k'}$.

\begin{lemma}\label{lem:field}
With the above notation, the functor $\otimes_k k'$ 
yields an exact functor 
\[ \uotimes_k k' : \ucC_k \longrightarrow \ucC_{k'}. \]
\end{lemma}

\begin{proof}
The composite functor 
$Q_{k'} \circ \otimes_k k' : \cC_k \to \ucC_{k'}$ is exact and sends
every finite $k$-group to $0$; hence it factors through a unique
exact functor $\cC_k/\cF_k \to \ucC_{k'}$. This yields the existence
and exactness of $\uotimes_k k'$.
\end{proof}

\begin{lemma}\label{lem:des}
Let $k'$ be a purely inseparable field extension of $k$, 
and $G'$ a $k'$-group.

\begin{enumerate}

\item[{\rm (i)}] There exists a smooth $k$-group $G$
and an epimorphism $f: G' \to G_{k'}$ such that $\Ker(f)$ is
infinitesimal. 

\item[{\rm (ii)}]
If $G' \subset H_{k'}$ for some $k$-group $H$, then there exists 
a $k$-subgroup $G \subset H$ such that $G' \subset G_{k'}$ and 
$G_{k'}/G'$ is infinitesimal. 

\end{enumerate}

\end{lemma}

\begin{proof}
(i) Let $n$ be a positive integer and consider the $n$th relative 
Frobenius morphism
\[ F^n_{G'/{k'}} : G' \longrightarrow G'^{(p^n)}. \]
Recall that the quotient
$G'/\Ker(F^n_{G'/{k'}})$ is smooth for $n \gg 0$. Since $\Ker(F^n_{G'/{k'}})$ 
is infinitesimal, we may assume that $G'$ is smooth. Then $F^n_{G'/{k'}}$ 
is an epimorphism in view of \cite[VIIA, Cor.~8.3.1]{SGA3}.

Next, note that $G'$ is defined over some finite subextension 
$k''$ of $k'$, i.e., there exists a $k''$-subgroup $G''$
such that $G' = G'' \otimes_{k''} k'$. By transitivity of base
change, we may thus assume that $k'$ is finite over $k$. Let
$q := [k':k]$, then $q = p^n$, where $p = \car(k)$ and $n$ is a 
positive integer; also, ${k'}^q \subset k$. Consider again the
morphism $F^n_{G'/{k'}}$; then by construction, 
$G'^{(p^n)} \cong G' \otimes_{k'} k'$, where $k'$ is sent to itself 
via the $q$th power map. Thus, $G'^{(p^n)} \cong G_{k'}$, where 
$G$ denotes the $k$-group $G' \otimes_{k'} k$; here $k'$ is sent to
$k$ via the $q$th power map again. So the induced map
$G' \to G_{k'}$  is the desired morphism.

(ii) As above, we may reduce to the case where $k'$ is finite
over $k$. Then the statement follows by similar arguments, 
see \cite[Lem.~4.3.5]{Brion-II} for details.
\end{proof}

We now are ready to prove Theorem \ref{thm:all} (iv):

\begin{theorem}\label{thm:insep}
Let $k'$ be a purely inseparable field extension of $k$. Then
the base change functor $\uotimes_k k' : \ucC_k \to \ucC_{k'}$
is an equivalence of categories.
\end{theorem}

\begin{proof}
By Lemma \ref{lem:des}, every $k'$-group $G'$ is isogenous to $G_{k'}$
for some smooth $k$-group $G$. It follows that $\uotimes_k k'$
is essentially surjective.

Next, let $G,H$ be smooth connected $k$-groups, and
$\uf : G \to H$ a $\ucC_k$-morphism, represented by 
a $\cC_k$-morphism $f : G \to H/H'$ for some finite $k$-subgroup
$H' \subset H$. If $\uf_{k'} : G_{k'} \to H_{k'}$ is zero in $\ucC_{k'}$,
then the image of $f_{k'} : G_{k'} \to H_{k'}/H'_{k'}$ is finite.
By Lemma \ref{lem:epi}, it follows that $f_{k'} = 0$.
This shows that $\uotimes_k k'$ is faithful.

We now check that $\uotimes_k {k'}$ is full. Let again $G,H$ 
be smooth connected $k$-groups, and let 
$\uf \in \Hom_{\ucC_{k'}}(G_{k'},H_{k'})$.
We show that there exists a finite $k$-subgroup 
$H' \subset H$ and a $k$-morphism $\varphi : G \to H/H'$ such
that $\varphi_{k'}$ represents $\uf$. For this, we may replace $H$
with its quotient by any finite $k$-subgroup.

Choose a representative $f : G_{k'} \to H_{k'}/H''$ of $\uf$, where
$H'' \subset H_{k'}$ is a finite $k'$-subgroup.
By Lemma \ref{lem:des}, there exists a $k$-subgroup $I \subset H$ 
such that $H'' \subset I_{k'}$ and $I_{k'}/H''$ is finite; then $I_{k'}$ is
finite as well, and hence so is $I$. We may thus replace $H$ by
$H/I$, and $f$ by its composition with the quotient morphism 
$H_{k'}/H'' \to H_{k'}/I_{k'} = (H/I)_{k'}$. Then $\uf$ is represented by 
a morphism $f: G_{k'} \to H_{k'}$.

Consider the graph $\Gamma(f) \subset G_{k'} \times_{k'} H_{k'}$. 
By Lemma \ref{lem:des} again, there exists a $k$-subgroup
$\Delta \subset G \times H$ such that $\Gamma(f) \subset \Delta_{k'}$
and $\Delta_{k'}/\Gamma(f)$ is finite. Then the intersection
$\Delta_{k'} \cap (0 \times H_{k'})$ is finite, since 
$\Gamma(f) \cap (0 \times H_{k'})$ is zero. Thus, 
$\Delta \cap (0 \times H)$ is finite as well; equivalently,
the $k$-group $H' := H \cap (0 \times \id)^{-1}(\Delta)$ is finite.
Denoting by $\Gamma$ the image of $\Delta$ in 
$G \times H/H'$, we have a cartesian square
\[ \CD
\Delta @>>> G \times H \\
@VVV @VVV \\
\Gamma @>>> G \times H/H', \\
\endCD \]
where the horizontal arrows are closed immersions, and the left
(resp.~right) vertical arrow is the quotient by 
$\Delta \cap (0 \times H)$ (resp.~by $H'$ acting on $H$ via 
addition). So $\Gamma$ is a $k$-subgroup of $G \times H/H'$,
and $\Gamma \cap (0 \times H/H')$ is zero; in other words,
the projection $\pi : \Gamma \to G$ is a closed immersion. Since
$G$ is smooth and connected, and 
$\dim(\Gamma) = \dim(\Delta) = \dim \Gamma(f) = \dim(G)$, it
follows that $\pi$ is an isomorphism. In other words, $\Gamma$
is the graph of a $k$-morphism $\varphi : G \to H/H'$. 
Since the above cartesian square lies in a push-out diagram,
\[ \CD
0 @>>> \Delta @>>> G \times H @>{\varphi - q}>> H/H' @>>> 0\\
& & @VVV @VVV @V{\id}VV \\
0 @>>> \Gamma @>>> G \times H/H' @>{\varphi - \id}>> H/H' @>>> 0, \\
\endCD \]
where $q : H \to H/H'$ denotes the quotient morphism, 
it follows that $\Delta = \ker(\varphi - q)$.
As $\Gamma(f) \subset \Delta_{k'}$, we see that 
$\varphi_{k'} = q_{k'} \circ f$. Thus, $\uf$ is represented by 
$\varphi_{k'}$; this completes the proof of the fullness assertion.
\end{proof}

\begin{remarks}\label{rem:fe}
(i) Likewise, the base change functor induces equivalences
of  categories $\ucU_k \to \ucU_{k'}$, $\ucT_k \to \ucT_{k'}$, 
$\ucL_k \to \ucL_{k'}$, and $\ucA_k \to \ucA_{k'}$. 
For tori, this follows much more directly from the anti-equivalence 
of $\ucT_k$ with the category of rational representations of the 
absolute Galois group of $k$, see Proposition \ref{prop:tss}.

(ii) In particular, the category $\ucU_k$ is equivalent to 
$\ucU_{k_i}$, where $k_i$ denotes the perfect closure of $k$ in 
$\bar{k}$. Recall from \cite[Thm.~V.1.4.3, Cor.~V.1.4.4]{DG} 
that the category $\cU_{k_i}$ is anti-equivalent to the category 
of finitely generated modules over the Dieudonn\'e ring 
$\bD = \bD_{k_i}$ which are killed by some power of the Verschiebung 
map $V$. Moreover, the category $\ucU_{k_i}$ is anti-equivalent 
to the category of finitely generated modules over the localization 
$\bD_{(V)}$ which are killed by some power of $V$; see 
\cite[\S V.3.6.7]{DG}.

By work of Schoeller, Kraft, and Takeuchi (see 
\cite{Schoeller, Kraft, Takeuchi}), the category $\cU_k$ is 
anti-equivalent to a category of finitely generated modules 
over a certain $k$-algebra, which generalizes the Dieudonn\'e 
ring but seems much less tractable.
\end{remarks}

\section{Tori, abelian varieties, and homological dimension}
\label{subsec:tavhd}

\subsection{Tori}
\label{subsec:t}

Denote by $\cT$ (resp.~$\cM$, $\cF \cM$) the full subcategory of $\cC$ 
with objects the tori (resp.~the groups of multiplicative type,
the finite groups of multiplicative type). Then $\cT$ is stable under
taking quotients and extensions, but not subobjets; in particular,
$\cT$ is an additive subcategory of $\cC$, but not an abelian
subcategory. Also, $\cF$ and $\cF \cM$ are stable under taking subobjects, 
quotients and extensions. Thus, we may form the quotient abelian category 
$\cM/\cF \cM$, as in \S \ref{subsec:dfp}. One may readily check 
that $\cM/\cF \cM$ is a full subcategory of $\cC/\cF$.

Let $\ucT$ be the full subcategory of $\cM/\cF \cM$ with objects
the tori. Since these are the smooth connected objects of $\cM$, 
one may check as in Lemma \ref{lem:equi} that
the inclusion of $\ucT$ in $\cM/\cF \cM$ is an equivalence of 
categories. The remaining statements of Lemma \ref{lem:equi}
also adapt to this setting; note that we may replace the
direct limits over all finite subgroups with those over all 
$n$-torsion subgroups, since tori are divisible. Also, Proposition 
\ref{prop:homdiv} yields natural isomorphisms
\[ \bQ \otimes_\bZ \Hom_\cC(T_1,T_2) 
\stackrel{\cong}{\longrightarrow} \Hom_{\ucT}(T_1,T_2) \]
for any tori $T_1,T_2$. 

By assigning with each group of multiplicative type $G$
its character group, 
\[ X(G):= \Hom_{k_s}(G_{k_s},\bG_{m,k_s}), \] 
one obtains an anti-equivalence between $\cM$ (resp.~$\cF \cM$) 
and the category of finitely generated (resp.~finite) abstract
commutative groups equipped with the discrete topology and
a continuous action of the Galois group $\Gamma$; see 
\cite[Thm.~IV.1.3.6]{DG}. Thus, the assignment 
\[ X_\bQ : G \longmapsto \bQ \otimes_{\bZ} X(G) =: X(G)_\bQ \]
yields a contravariant exact functor from $\cM$ to the 
category $\Rep_\bQ(\Gamma)$ of finite-dimensional $\bQ$-vector 
spaces equipped with a continuous representation of $\Gamma$
as above; moreover, every finite group of multiplicative type 
is sent to $0$. This yields in turn a contravariant exact functor
\[ \uX_\bQ : \ucT \longrightarrow \Rep_\bQ(\Gamma). \]

\begin{proposition}\label{prop:tss}
The functor $\uX_\bQ$ is an anti-equivalence of categories. 
In particular, the category $\ucT$ is semi-simple, and 
$\Hom_{\ucT}(T_1,T_2)$ is a finite-dimensional $\bQ$-vector space 
for any tori $T_1,T_2$.
\end{proposition}

\begin{proof}
Given a finite-dimensional $\bQ$-vector space $V$ equipped
with a continuous action of $\Gamma$, there exists a
finitely generated $\Gamma$-stable subgroup $M \subset V$
which spans $V$; thus, $V \cong X(T)_\bQ$, where $T$ denotes
the torus with character group $M$. So $\uX_\bQ$ is essentially
surjective.

Given two tori $T_1,T_2$, the natural isomorphism
$\Hom_\cM(T_1,T_2) \cong \Hom^\Gamma(X(T_2),X(T_1))$
yields an isomorphism
\[ \Hom_{\ucT}(T_1,T_2) \cong \Hom^\Gamma(X(T_2)_\bQ, X(T_1)_\bQ) \]
in view of Proposition \ref{prop:homdiv}. It follows that 
$\uX_\bQ$ is fully faithful.
\end{proof}

\begin{lemma}\label{lem:maxt}
\begin{enumerate}

\item[{\rm (i)}]
Every algebraic group $G$ has a unique maximal torus, $T(G)$. 

\item[{\rm (ii)}]
Every morphism of algebraic groups $u : G \to H$ sends $T(G)$ to $T(H)$. 

\item[{\rm (iii)}]
The formation of $T(G)$ commutes with base change under field extensions.

\end{enumerate}

\end{lemma}

\begin{proof}
(i) This follows from the fact that $T_1 + T_2$ 
is a torus for any subtori $T_1,T_2 \subset G$. 

(ii) Just note that the image of a torus under any morphism is still
a torus.

(iii) Consider an algebraic group $G$, its maximal torus $T$,
and a field extension $k'$ of $k$. If $\car(k) = 0$, then Theorem 
\ref{thm:zero} implies that $G/T$ is a an extension of an abelian 
variety by a product $M \times U$, where $M$ is finite and $U$ 
unipotent. As a consequence, a similar assertion holds for
$G_{k'}/T_{k'}$; it follows that $G_{k'}/T_{k'}$ contains no non-zero 
torus, and hence $T_{k'}$ is the maximal torus of $G_{k'}$.
On the other hand, if $\car(k) > 0$, then $G/T$ is a $3$-step 
extension of a unipotent group by an abelian variety by a finite 
group, in view of Theorem \ref{thm:pos}. It follows similarly 
that $T_{k'}$ is the maximal torus of $G_{k'}$.
\end{proof}

By Lemma \ref{lem:maxt}, the assignment $G \mapsto T(G)$ yields 
a functor 
\[ T : \cC \longrightarrow \cT, \]
the \emph{functor of maximal tori}. This functor is not exact, 
as seen from the exact sequence 
\[ 0 \longrightarrow G[n] \longrightarrow G 
\stackrel{n_G}{\longrightarrow} G \longrightarrow 0, \]
where $G$ is a non-zero torus and $n$ a non-zero integer.
But $T$ is exact up to finite groups, as shown by the following;

\begin{lemma}\label{lem:max}
Every exact sequence in $\cC$ 
\[ 0 \longrightarrow G_1 \stackrel{u}{\longrightarrow} G_2 
\stackrel{v}{\longrightarrow} G_3 \longrightarrow 0 \]
yields a complex in $\cC$
\[ 0 \longrightarrow T(G_1) \stackrel{T(u)}{\longrightarrow} T(G_2) 
\stackrel{T(v)}{\longrightarrow} T(G_3) \longrightarrow 0, \]
where $T(u)$ is a monomorphism, $T(v)$ an epimorphism, and
$\Ker \, T(v)/\Im \, T(u)$ is finite. 
\end{lemma}

\begin{proof}
We argue as in the proof of Theorem \ref{thm:pos} (iii).
Clearly, $T(u)$ is a monomorphism. Also, the group
\[ \Ker \, T(v)/\Im \, T(u) = T(G_2) \cap u(G_1)/u(T(G_1)) \]
is the quotient of a group of multiplicative type by its 
maximal torus, and hence is finite.

To show that $T(v)$ is an epimorphism, we may replace $G_2$ with 
$v^{-1}(T(G_3))$, and hence assume that $G_3$ is a torus. Next, 
we may replace $G_2$ with $G_2/T(G_2)$, and hence assume that 
$T(G_2)$ is zero. We then have to check that $G_3$ is zero.

If $\car(k) = 0$, then there is an exact sequence 
\[ 0 \longrightarrow M \times U \longrightarrow G_2 
\longrightarrow A \longrightarrow 0 \]
as in Theorem \ref{thm:zero}, where $M$ is finite. Thus, every 
morphism $G_2 \to G_3$ has finite image. Since $v : G_2 \to G_3$
is an epimorphism, it follows that $G_3 = 0$.
On the other hand, if $\car(k) > 0$, then there are exact sequences 
\[ 0 \longrightarrow H \longrightarrow G_2 
\longrightarrow U \longrightarrow 0, \quad  
0 \longrightarrow M \longrightarrow H 
\longrightarrow A \longrightarrow 0 \]
as in Theorem \ref{thm:pos}, where $M$ is finite. This implies again 
that every morphism $G_2 \to G_3$ has finite image, and hence
that $G_3 = 0$.
\end{proof}

\begin{proposition}\label{prop:proj}

\begin{enumerate}

\item[{\rm (i)}]  
The functor of maximal tori yields an exact functor 
\[ \uT : \ucC \longrightarrow \ucT, \] 
right adjoint to the inclusion $\ucT \to \ucC$. Moreover, 
$\uT$ commutes with base change under field extensions.

\item[{\rm (ii)}]  
Every torus is a projective object in $\ucC$.

\end{enumerate}

\end{proposition}

\begin{proof}
(i) Composing $T$ with the functor $\cT \to \ucT$ induced by
the quotient functor $Q$, we obtain an exact functor $\cC \to \ucT$ 
(Lemma \ref{lem:max}), which sends every finite group to $0$.
This yields an exact functor $\uT : \ucC \to \ucT$.
The adjointness assertion follows from the natural isomorphism
\[ \Hom_{\cC}(T,G) \cong \Hom_{\cT}(T,T(G)) \]
for any torus $T$ and any algebraic group $G$,
which yields a natural isomorphism
\[ \Hom_{\ucC}(T,G) \cong \Hom_{\ucT}(T,T(G)) \]
by using Lemmas \ref{lem:equi} and \ref{lem:max}.
Finally, the assertion on field extensions is a direct
consequence of Lemma \ref{lem:maxt}. 

(ii) This follows by arguing as in the proof of Proposition
\ref{prop:ru} (iii).
\end{proof}

\subsection{Abelian varieties}
\label{subsec:av}

Denote by $\cA$ (resp.~$\cP$) the full subcategory of $\cC$ 
with objects the abelian varieties (resp.~the proper groups, i.e., 
those algebraic groups $G$ such that the structure map
$G \to \Spec(k)$ is proper). Like the categorye of tori, 
$\cA$ is stable under taking quotients and extensions, but not 
subobjects; so $\cA$ is an additive subcategory of $\cC$, but not
an abelian subcategory. Also, $\cP$ is stable under taking subobjects, 
quotients and extensions; it also contains the category $\cF$ of 
finite groups. We may thus form the quotient abelian category 
$\cP/\cF$, which is a full subcategory of $\cC/\cF$.

Next, let $\ucA$ be the full subcategory of $\cP/\cF$ with objects
the abelian varieties. As in \S \ref{subsec:t}, the inclusion of $\ucA$ 
in $\cP/\cF$ is an equivalence of categories, and the remaining 
statements of Lemma \ref{lem:equi} adapt to this setting. Also,
Proposition \ref{prop:homdiv} yields natural isomorphisms 
\[ \bQ \otimes_\bZ \Hom_\cC(A_1,A_2) \stackrel{\cong}{\longrightarrow}
\Hom_{\ucA}(A_1,A_2) \]
for any abelian varieties $A_1,A_2$. Since the abelian group
$\Hom_\cC(A_1,A_2)$ has finite rank (see \cite[Thm.~12.5]{Milne86}),
$\Hom_{\ucA}(A_1,A_2)$ is a finite-dimensional $\bQ$-vector space. 
Moreover, the category $\ucA$ is semi-simple, in view of the Poincar\'e 
complete reducibility theorem (which holds over an arbitrary field, 
see \cite[Cor.~3.20]{Conrad} or \cite[Cor.~4.2.6]{Brion-II}).

\begin{lemma}\label{lem:albq}

\begin{enumerate}

\item[{\rm (i)}] 
Every smooth connected algebraic group $G$ has a largest 
abelian variety quotient, 
\[ \alpha = \alpha_G : G \longrightarrow A(G). \]
Moreover, $\Ker(\alpha)$ is linear and connected. 

\item[{\rm (ii)}] 
Every morphism $u : G \to H$, where $H$ is smooth and connected, 
induces a unique morphism $A(u) : A(G) \to A(H)$ such that 
the square
\[ \CD
G @>{u}>> H \\
@V{\alpha_G}VV @V{\alpha_H}VV \\
A(G) @>{A(u)}>> A(H)\\
\endCD \]
commutes. 

\item[{\rm (iii)}] For any field extension $k'$ of $k$, 
the natural morphism $A(G_{k'}) \to A(G)_{k'}$ is an isomorphism
if $\car(k) = 0$, and an isogeny if $\car(k) > 0$. 

\end{enumerate}

\end{lemma}

\begin{proof}
(i) and (ii) The assertions are direct consequences of Theorem 
\ref{thm:che} (ii).

(iii) By (i), we have an exact sequence
\[ 0 \longrightarrow L(G) \longrightarrow G 
\stackrel{\alpha}{\longrightarrow} A(G)
\longrightarrow 0, \]
where $L(G)$ is linear and connected.
This yields an exact sequence
\[ 0 \longrightarrow L(G)_{k'} \longrightarrow G_{k'} 
\longrightarrow A(G)_{k'} \longrightarrow 0. \]
Thus, $L(G_{k'}) \subset L(G)_{k'}$ and we obtain an exact sequence
\[ 0 \longrightarrow L(G)_{k'}/L(G_{k'}) \longrightarrow A(G_{k'}) 
\longrightarrow A(G)_{k'} \longrightarrow 0. \]
Sinc $L(G)_{k'}$ is linear, the quotient $L(G)_{k'}/L(G_{k'})$ must be
finite; this yields the assertion when $\car(k) > 0$.

When $\car(k) = 0$, we may characterize $L(G)$ as the largest
connected linear subgroup of $G$. It follows that  
$L(G)_{k'} \subset L(G_{k'})$; hence equality holds, and 
$A(G_{k'}) \stackrel{\cong}{\rightarrow} A(G)_{k'}$.
\end{proof}

\begin{remarks}\label{rem:albq}
(i) An arbitrary algebraic group $G$ may admit no largest
abelian variety quotient, as shown by the following
variant of \cite[Ex.~4.3.8]{Brion-II}: let $A$ be a non-zero 
abelian variety, and choose an integer $n \geq 2$. Let 
\[ G := (A \times A[n^2])/ \diag(A[n]), \]
where $A[n]$ is viewed as a subgroup of $A[n^2]$. Consider
the subgroups of $G$
\[ H_1 :=  (A[n] \times A[n^2])/ \diag(A[n]), \quad
H_2 := \diag(A[n^2])/ \diag(A[n]). \]
Then $G/H_1$, $G/H_2$ are both isomorphic to $A$. Also,
$H_1 \cap H_2 = 0$, and $G$ is not an abelian variety. 

(ii) The assignment $G \mapsto A(G)$ does not preserve exactness
of sequences of smooth connected algebraic groups. For example,
consider an elliptic curve $E$ equipped with a $k$-rational
point of prime order $\ell \geq 2$. Assume that $k$ contains
a nontrivial $\ell$th root of unity; this identifies $\mu_\ell$
with the constant group scheme $\bZ/\ell \bZ$. Consider the
quotient 
\[ G := (E \times \bG_m)/\diag(\mu_\ell), \] 
with an obvious notation. Then $G$ is a smooth connected algebraic 
group, which lies in an exact sequence
\[ 0 \longrightarrow E \longrightarrow G \longrightarrow
\bG_m \longrightarrow 0. \]
Moreover, the induced map $A(E) \to A(G)$ is just the
quotient map $E \to E/\mu_\ell$.
\end{remarks}

We now show that the assignment $G \mapsto A(G)$ is exact up 
to finite groups:

\begin{lemma}\label{lem:alb}
Consider an exact sequence in $\cC$
\[ 0 \longrightarrow G_1 
\stackrel{u}{\longrightarrow} G_2
\stackrel{v}{\longrightarrow} G_3 \longrightarrow 0 , \]
where $G_1,G_2,G_3$ are smooth and connected. Then we have 
a commutative diagram in $\cC$
\[ \CD
0 @>>> G_1 @>{u}>> G_2 @>{v}>> G_3 @>>> 0 \\
& & @V{\alpha_1}VV @V{\alpha_2}VV @V{\alpha_3}VV \\
0 @>>> A(G_1) @>{A(u)}>> A(G_2) @>{A(v)}>> A(G_3) @>>> 0, \\
\endCD \]
where $A(v)$ is an epimorphism, and $\Ker \, A(u)$, 
$\Ker \, A(v)/\Im \, A(u)$ are finite.
\end{lemma}

\begin{proof}
Clearly, $A(v)$ is an epimorphism. Let 
$L_i := \Ker(\alpha_i)$ for $i = 1,2,3$; then each $L_i$ 
is connected and linear by Theorem \ref{thm:che}. 
We have isomorphisms $\Ker \, A(u) \cong u^{-1}(L_2)/L_1$,
$\Im \, A(u) \cong u(G_1)/u(G_1) \cap L_2$ and 
$\Ker \, A(v) \cong v^{-1}(L_3)/L_2$. Since $u^{-1}(L_2)$
is linear, $\Ker \, A(u)$ is linear as well; it is also proper, 
and hence finite. Also, 
\[ \Ker \, A(v)/\Im \, A(u) \cong v^{-1}(L_3)/L_2 + u(G_1) \]
is a quotient of $v^{-1}(L_3)/u(G_1) \cong L_3$. It follows similarly
that $\Ker \, A(v)/\Im \, A(u)$ is finite.
\end{proof}

Next, we obtain a dual version of Proposition \ref{prop:proj}:

\begin{proposition}\label{prop:inj}

\begin{enumerate}

\item[{\rm (i)}]  
The Albanese functor yields an exact functor 
\[ \uA : \ucC \longrightarrow \ucA, \] 
which is left adjoint to the inclusion $\ucA \to \ucC$. Moreover, 
$\uA$ commutes with base change under field extensions.

\item[{\rm (ii)}]  
Every abelian variety is an injective object in $\ucC$.

\end{enumerate}

\end{proposition}

\begin{proof}
(i) Consider a $\ucC$-morphism $\uf : G \to H$ and choose 
a representative by a $\cC$-morphism $f : G \to H/H'$,
where $H'$ is a finite subgroup of $H$. This yields an
$\cA$-morphism
\[ A(f): A(G) \longrightarrow A(H/H') = A(H)/\alpha_H(H'), \]
and hence an $\ucA$-morphism 
$\uA (f) : A(G) \to A(H)$. One may readily check that
$\uA(f)$ depends only on $\uf$, in a covariant way. By 
Proposition \ref{prop:exa} and Lemma \ref{lem:albq}, the resulting
functor $\uA$ is exact and commutes with base change under field
extensions.

To show the adjointness assertion, consider a smooth connected
algebraic group $G$ and an abelian variety $A$. Then the map
\[ \Hom_{\cA}(A(G),A) \longrightarrow \Hom_{\cC}(G,A), \quad
f \longmapsto f \circ \alpha_G \]
is an isomorphism by Lemma \ref{lem:albq} again. In view of
Proposition \ref{prop:homdiv}, it follows that the analogous map 
\[ \Hom_{\ucA}(A(G),A) \longrightarrow \Hom_{\ucC}(G,A) \]
is an isomorphism as well.

(ii) This follows formally from (i).
\end{proof}

\begin{remarks}\label{rem:av}
(i) Denote by $\wA$ the dual of an abelian variety $A$. Then the
assignement $A \mapsto \wA$ yields a contravariant endofunctor of
$\cA$, which is involutive and preserves isogenies and finite
products. As an easy consequence, we obtain a contravariant endofunctor 
of $\ucA$, which is involutive and exact. Note that each abelian
variety is (non-canonically) $\ucA$-isomorphic to its dual, via the
choice of a polarization.

(ii) Let $k'$ be a field extension of $k$. Then the assignement
\[ A \longmapsto \bQ \otimes_{\bZ} A(k') =: A(k')_{\bQ} \]
yields a functor from $\cA$ to the category of $\bQ$-vector spaces
(possibly of infinite dimension), which preserves finite products. 
Moreover, each isogeny 
\[ f: A \longrightarrow B \] 
yields an isomorphism
\[ A(k')_{\bQ} \stackrel{\cong}{\longrightarrow} B(k')_{\bQ}, \] 
since this holds for the multiplication maps $n_A$. Thus, 
the above assignement yields an exact functor from $\ucA$ to 
the category of $\bQ$-vector spaces.
\end{remarks}

\subsection{Vanishing of extension groups}
\label{subsec:ves}

In this subsection, we prove the assertion (v) of Theorem 
\ref{thm:all}: $\hd(\ucC) = 1$. We first collect general 
vanishing results for extension groups in $\ucC$:

\begin{lemma}\label{lem:list}
Let $G$ be a smooth connected algebraic group, $U$ a smooth
connected unipotent group, $A$ an abelian variety, and $T$ a torus.

\begin{enumerate}

\item[{\rm (i)}] $\Ext_{\ucC}^n(T,G) = 0 = \Ext_{\ucC}^n(G,A) = 0$ 
for all $n \geq 1$.

\item[{\rm (ii)}] If $\car(k) = 0$, then $\Ext_{\ucC}^n(U,G) = 0$
for all $n \geq 1$.

\item[{\rm (iii)}] If $\car(k) = p > 0 $ and $G$ is divisible, 
then $\Ext_{\ucC}^n(U,G) = 0 = \Ext_{\ucC}^n(G,U)$ for all $n \geq 0$.

\item[{\rm (iv)}] If $G$ is linear, then $\Ext_{\ucC}^n(G,T) = 0$
for all $n \geq 1$. If in addition $\car(k) = 0$, then
$\Ext_{\ucC}^n(G,U) = 0$ for all $n \geq 1$ as well.

\end{enumerate}

\end{lemma}

\begin{proof}
(i) Just recall that $T$ is projective in $\ucC$ (Proposition
\ref{prop:proj}), and $A$ is injective in $\ucC$ (Proposition
\ref{prop:inj}).

(ii) Likewise, $U$ is projective in $\ucC$ by Proposition
\ref{prop:ru}.

(iii) Since $U$ is unipotent, there exists a positive integer 
$m$ such that $p^m_G = 0$. It follows that both groups 
$\Ext^m_{\ucC}(U,G)$ and $\Ext^m_{\ucC}(G,U)$ are $p^m$-torsion 
(see e.g. \cite[Lem.~I.3.1]{Oort66}). But these groups are also 
modules over $\End_{\ucC}(G)$, and hence $\bQ$-vector spaces 
by Proposition \ref{prop:homdiv}. This yields the assertion. 

(iv) By Proposition \ref{prop:el} and the long exact sequence for
Ext groups, we may assume that $G$ is unipotent or a torus. 
In the latter case, both assertions follows from (i); in the former
case, the first assertion follows from (ii) and (iii), and the 
second assertion, from the fact that unipotent groups are just
vector spaces. 
\end{proof}

Next, recall that $\hd(\ucC) \geq 1$ in view of Examples 
\ref{ex:nonsplit}. So, to show that $\hd(\ucC) = 1$, 
it suffices to check the following:

\begin{lemma}\label{lem:van}
For any smooth connected algebraic groups $G,H$ and any integer
$n \geq 2$, we have $\Ext_{\ucC}^n(G,H) = 0$.
\end{lemma}

\begin{proof}
Let $\eta \in \Ext_{\ucC}^n(G,H)$, where $n \geq 3$.
Then $\eta$ is represented by an exact sequence in $\ucC$
\[ 0 \longrightarrow H \longrightarrow G_1 \longrightarrow 
\cdots \longrightarrow G_n \longrightarrow G \longrightarrow 0, \]
which we may cut into two exact sequences in $\ucC$
\[ 0 \longrightarrow H \longrightarrow G_1 \longrightarrow G_2 
\longrightarrow K \longrightarrow 0,
\quad
0 \longrightarrow K \longrightarrow G_3 \longrightarrow 
\cdots \longrightarrow G_n \longrightarrow G \longrightarrow 0. \]
Thus, $\eta$ can be written as a Yoneda product 
$\eta_1 \cup \eta_2$, where $\eta_1 \in \Ext^2_{\ucC}(G,K)$ and 
$\eta_2 \in \Ext^{n-2}_{\ucC}(K,H)$. So it suffices to show 
the assertion when $n = 2$. 

Using the long exact sequences for $\Ext$ groups, we may 
further reduce to the case where $G,H$ are simple objects 
in $\ucC$, i.e., $\bG_a$, simple tori $T$, or simple abelian 
varieties $A$ (Proposition \ref{prop:simple}). In view of
Lemma \ref{lem:list}, it suffices in turn to check that

\begin{enumerate}

\item[(i)]
$\Ext^2_{\ucC}(A,T) = 0$, 

\item[(ii)]
$\Ext^2_{\ucC}(A,\bG_a) = 0$ when $\car(k) = 0$, 

\item[(iii)]
$\Ext^2_{\ucC}(\bG_a,\bG_a) = 0$ when $\car(k) > 0$.

\end{enumerate}

For (i), we adapt the argument of \cite[Prop.~II.12.3]{Oort66}.
Let $\eta \in \Ext^2_{\ucC}(A,T)$ be represented by an exact
sequence in $\ucC$
\[ 0 \longrightarrow T \stackrel{u}{\longrightarrow} G_1 
\stackrel{v}{\longrightarrow} G_2 
\stackrel{w}{\longrightarrow} A \longrightarrow 0. \]
As above, $\eta = \eta_1 \cup \eta_2$, where $\eta_1$ denotes
the class of the extension 
\[ 0 \longrightarrow T \stackrel{u}{\longrightarrow} G_1 
\longrightarrow K \longrightarrow 0, \]
and $\eta_2$ that of the extension
\[ 0 \longrightarrow K \longrightarrow G_2 
\stackrel{w}{\longrightarrow} A \to 0. \]
Also, note that $\uA(w) : \uA(G_2) \to A$ is an epimorphism in
$\ucA$, and hence has a section, say $s$. Denoting by 
$H_2$ the pull-back of $s(A)$ under the Albanese morphism
$G_2 \to \uA(G_2)$, we have a monomorphism $\iota : H_2 \to G_2$
in $\ucC$, such that $\uA(w \circ \iota) : \uA(H_2) \to A$
is an isomorphism. This yields a commutative diagram of exact 
sequences in $\ucC$
\[ \CD
0 @>>> T @>>> H_1 @>>> H_2 @>>> A @>>> 0 \\
& & @V{\id}VV @VVV  @V{\iota}VV @V{\id}VV \\
0 @>>> T @>{u}>> G_1 @>{v}>> G_2 @>{w}>> A @>>> 0. \\
\endCD \]
Thus, $\eta$ is represented by the top exact sequence,
and hence we may assume that $\uA(w)$ is an isomorphism.
Then $K$ is linear in view of Theorem \ref{thm:che}.
Thus, $\kappa_1 = 0$ by Lemma \ref{lem:list} (iv). So
$\eta = 0$; this completes the proof of (i).

For (ii), we replace $T$ with $\bG_a$ in the above argument,
and use the vanishing of $\Ext^1_{\ucC}(L,\bG_a)$ for $L$ linear 
(Lemma \ref{lem:list} (iv) again).

Finally, for (iii), it suffices to show that 
$\Ext^2_{\ucU}(\bG_a,\bG_a) = 0$. 
Also, we may assume that $k$ is perfect, in view of 
Theorem \ref{thm:insep}. Let $\eta \in \Ext^2_{\ucU}(\bG_a,\bG_a)$
be represented by an exact sequence in $\ucU$,
\[ 0 \longrightarrow \bG_a \longrightarrow G_1 
\longrightarrow G_2 \longrightarrow \bG_a \longrightarrow 0. \]
Then Proposition \ref{prop:exa} yields an exact sequence in $\cU$,
\[ 0 \longrightarrow \bG_a \longrightarrow H_1 
\longrightarrow H_2 \longrightarrow \bG_a \longrightarrow 0, \]
and a commutative diagram in $\ucU$, 
\[ \CD
0 @>>> \bG_a @>>> G_1 @>>> G_2 @>>> \bG_a @>>> 0 \\
& & @VVV @VVV  @VVV @VVV \\
0 @>>> \bG_a @>>> H_1 @>>> H_2 @>>> \bG_a @>>> 0. \\
\endCD \]
where the vertical arrows are isomorphisms in $\ucU$ (here we use 
the fact that the quotient of  $\bG_a$ by a finite subgroup is 
isomorphic to $\bG_a$). Since $\Ext^2_{\cU}(\bG_a,\bG_a) = 0$ 
(see \cite[V.1.5.1, V.1.5.2]{DG}), the bottom exqct sequence is
equivalent to the trivial exact sequence in $\cU$, and hence
in $\ucU$. Thus $\eta = 0$ as desired.
\end{proof}

\begin{remark}\label{rem:van}
When $k$ is perfect, the groups $\Ext^n_{\cC}(G,H)$ 
are torsion for all $n \geq 2$ and all algebraic groups $G,H$,
in view of \cite[Cor., p.~439]{Milne70}. In fact, this assertion extends
to an arbitrary field $k$: indeed, it clearly holds 
when $G$ or $H$ is finite, or more generally $m$-torsion for 
some positive integer $m$. Using Proposition \ref{prop:el}, 
one may thus reduce to the case when $G,H$ are simple 
objects of $\ucC$. Then the assertion is obtained by combining 
Proposition \ref{prop:homdiv}, Lemma \ref{lem:list}, and 
the proof of Lemma \ref{lem:van}.

\end{remark}

\section{Structure of isogeny categories}
\label{sec:sic}

\subsection{Vector extensions of abelian varieties}
\label{subsec:veav}

In this subsection, we assume that $\car(k) = 0$. 
Recall that a \emph{vector extension} of an abelian variety $A$ 
is an algebraic group $G$ that lies in an extension
\[ \xi : \quad 0 \longrightarrow U \longrightarrow G 
\longrightarrow A \longrightarrow 0, \]
where $U$ is unipotent. Then $U = R_u(G)$ and $A = A(G)$
are uniquely determined by $G$; also, the extension $\xi$ has 
no non-trivial automorphisms, since $\Hom_\cC(A,U) = 0$. Thus, 
the data of the algebraic group $G$ and the extension $\xi$ 
are equivalent.

We denote by $\cV$ the full subcategory of $\cC$ with objects 
the vector extensions (of all abelian varieties). By Theorems
\ref{thm:che} and \ref{thm:lin}, the objects of $\cV$ are exactly 
those smooth connected algebraic groups that admit no non-zero
subtorus. In view of Lemmas \ref{lem:maxt} and \ref{lem:max}, 
this readily implies:

\begin{lemma}\label{lem:vec}

\begin{enumerate}

\item[{\rm (i)}] Let $0 \to G_1 \to G_2 \to G_3 \to 0$
be an exact sequence in $\cC$, where $G_2$ is connected. 
Then $G_2$ is an object of $\cV$ if and only if so are $G_3$ and $G_1^0$.

\item[{\rm (ii)}] Let $f : G \to H$ be an isogeny of connected 
algebraic groups. Then $G$ is an object of $\cV$ if and only if 
so is $H$.

\item[{\rm (iii)}] Let $k'$ be a field extension of $k$, and $G$
an algebraic $k$-group. Then $G$ is an object of $\cV_k$ if and only
if $G_{k'}$ is an object of $\cV_{k'}$.
\end{enumerate}

\end{lemma}   

In particular, $\cV$ is stable under taking quotients and extensions,
but not subobjects; like $\cT$ and $\cA$, it is an additive
subcategory of $\cC$, but not an abelian subcategory.

Next, recall from \cite{Rosenlicht-II} or \cite[\S 1.9]{MM} that 
every abelian variety $A$ has a universal vector extension,
\[ \xi(A) : \quad 0 \longrightarrow U(A) \longrightarrow E(A) 
\longrightarrow A \longrightarrow 0, \]
where $U(A)$ is the additive group of the vector space
$H^1(A,\cO_A)^*$; moreover, $\dim U(A) = \dim A$. Also,
$E(A)$ is anti-affine, i.e., every morphism from 
$E(A)$ to a linear algebraic group is zero (see e.g. 
\cite[Prop.~5.5.8]{Brion-II}).

\begin{proposition}\label{prop:vec}

\begin{enumerate}

\item[{\rm (i)}] The assignments $A \mapsto E(A)$, $A \mapsto U(A)$ 
yield additive functors
\[ E: \cA \longrightarrow \cV, \quad 
U : \cA \longrightarrow \cU, \] 
which commute with base change under field extensions.

\item[{\rm (ii)}] For any morphism $f: A \to B$ of abelian varieties,
the map $U(f) : U(A) \to U(B)$ is the dual of the pull-back morphism 
$f^* : H^1(B,\cO_B) \to H^1(A,\cO_A)$. Moreover, $U(f)$ is zero 
(resp.~an isomorphism) if and only if $f$ is zero (resp.~an isogeny).

\item[{\rm (iii)}] $E$ is left adjoint to the Albanese functor 
$A: \cV \to \cA$.

\end{enumerate}

\end{proposition}

\begin{proof}
We prove (i) and (ii) simultaneously. 
Let $f: A \to B$ be a morphism of abelian varieties.
Consider the pull-back diagram of exact sequences
\[ \CD
0 @>>> U(B) @>>> F @>>> A @>>> 0 \\
& & @V{\id}VV @VVV @V{f}VV \\
0 @>>> U(B) @>>> E(B) @>>> B @>>> 0. \\
\endCD \]
The universal property of $\xi(A)$ yields
a commutative diagram of exact sequences
\[ \CD
0 @>>> U(A) @>>> E(A) @>>> A @>>> 0 \\
& & @VVV @VVV @V{\id}VV \\
0 @>>> U(B) @>>> F @>>> A @>>> 0, \\
\endCD \]
and hence another such diagram,
\[ \CD
0 @>>> U(A) @>>> E(A) @>>> A @>>> 0 \\
& & @V{U(f)}VV @V{E(f)}VV @V{f}VV \\
0 @>>> U(B) @>>> E(B) @>>> B @>>> 0, \\
\endCD \]
which defines morphisms $E(f)$ and $U(f)$. 

Next, let $\eta \in H^1(B,\cO_B)$, so that we have a push-out
diagram of extensions
\[ \CD
0 @>>> U(B) @>>> E(B) @>>> B @>>> 0 \\
& & @V{\eta}VV @VVV @V{\id}VV \\
0 @>>> \bG_a @>>> E_\eta @>>> B @>>> 0. \\
\endCD \]
By construction, the pull-back of $E_\eta$ by $f$ is the push-out 
of $\xi(A)$ by $\eta \circ U(f) : U(A) \to \bG_a$. Hence $U(f)$ 
is the dual of $f^* : H^1(B,\cO_B) \to H^1(A,\cO_A)$. As a
consequence, $U$ is a covariant functor, and hence so is $E$.

Since the formation of the universal vector extension commutes
with base change under field extensions, the functors $E$ and $U$ 
commute with such base change as well. Clearly, they are additive;
this completes the proof of (i).

To complete the proof of (ii), recall the canonical isomorphism 
$H^1(A,\cO_A) \cong \Lie(\widehat{A})$, where the right-hand side denotes 
the Lie algebra of the dual abelian variety 
(see \cite[Rem.~9.4]{Milne86}). This isomorphism identifies
$f^*$ with $\Lie(\hat{f})$, where $\hat{f}: \widehat{B} \to \widehat{A}$
denotes the dual morphism of $f$. As a consequence, 
\[ U(f) = 0 \Leftrightarrow f^* = 0 \Leftrightarrow \hat{f} = 0 
\Leftrightarrow f = 0, \]
where the second equivalence holds since $\car(k) = 0$, and the third one
follows from biduality of abelian varieties.
Likewise, $U(f)$ is an isomorphism if and only if $f^*$ is an isomorphism;
equivalently, $\hat{f}$ is an isogeny, i.e., $f$ is an isogeny.

(iii) Given a vector extension $0 \to U \to G \to A(G) \to 0$, 
we check that the map
\[ \alpha : \Hom_{\cV}(E(A), G) \longrightarrow 
\Hom_{\cA}(A, A(G)), \quad u \longmapsto A(u) \]
is an isomorphism.

Consider a morphism $u : E(A) \to G$ such that $A(u) = 0$. 
Then $u$ factors through a morphism $E(A) \to R_u(G)$, and 
hence $u = 0$ as $E(A)$ is anti-affine. 

Next, consider a morphism $v : A \to A(G)$. By (i), we have 
a commutative square 
\[ \CD
E(A) @>{E(v)}>> E(A(G)) \\
@VVV @VVV \\
A @>{v}>> A(G). \\
\endCD \]
Also, the universal property of $\xi(A)$ yields 
a commutative square
\[ \CD
E(A(G)) @>{\delta}>> G \\
@VVV @V{\alpha_G}VV \\
A(G) @>{\id}>> A(G). \\
\endCD \]
Thus, $w := E(v) \circ \delta \in \Hom_{\cV}(E(A),G)$ 
satisfies $\alpha(v) = u$.
\end{proof}

Denote by $\ucV$ the isogeny category of vector extensions, that is, 
the full subcategory of $\ucC$ with the same objects as $\cV$. 
Then $\ucV$ is an abelian category in view of Lemma \ref{lem:vec}.
Also, Proposition \ref{prop:homdiv} yields natural isomorphisms
\[ \bQ \otimes_\bZ \Hom_{\cC}(G_1,G_2) 
\stackrel{\cong}{\longrightarrow} \Hom_{\ucV}(G_1,G_2), \quad
\bQ \otimes_\bZ \Ext^1_{\cC}(G_1,G_2) 
\stackrel{\cong}{\longrightarrow} 
\Ext^1_{\ucV}(G_1,G_2) \]
for any objects $G_1,G_2$ of $\ucV$.

\begin{corollary}\label{cor:vec}

\begin{enumerate}

\item[{\rm (i)}] The functors $E: \cA \to \cV$, $U: \cA \to \cU$
yield exact functors 
\[ \uE : \ucA \longrightarrow \ucV, \quad 
\uU : \ucA \longrightarrow \cU, \]
which commute with base change under field extensions.
Moreover, $\uE$ is left adjoint to the Albanese functor 
$\uA : \ucV \to \ucA$.

\item[{\rm (ii)}] The universal vector extension of any abelian
variety is a projective object of $\ucV$. 

\end{enumerate}

\end{corollary}

\begin{proof}
(i) This follows from Propositions \ref{prop:exa} and \ref{prop:vec}.

(ii) We have canonical isomorphisms for any object $G$ of $\ucV$:
\[ \Hom_{\ucV}(E(A), G) \cong 
\bQ \otimes_{\bZ} \Hom_{\cV}(E(A), G) \cong
\bQ \otimes_{\bZ} \Hom_{\cA}(A, A(G)) \cong 
\Hom_{\ucA}(A, \uA(G)), \]
where the first and third isomorphisms follow from Proposition
\ref{prop:homdiv}, and the second one from Proposition \ref{prop:vec}
again. Since the Albanese functor $\uA$ is exact (Proposition 
\ref{prop:inj}), it follows that the functor 
$G \mapsto \Hom_{\ucV}(E(A), G)$ is exact as well.
\end{proof}

Next, let $G$ be an object of $\cV$.
Form and label the commutative diagram of exact sequences in $\cC$
\[ \CD
0 @>>> U(A) @>{\iota}>> E(A) @>>> A @>>> 0 \\
& & @V{\gamma}VV @V{\delta}VV @V{\id}VV \\
0 @>>> U @>>> G @>>> A @>>> 0, \\
\endCD \]
where $U = U(G)$, $A = A(G)$, and $\gamma = \gamma_G$ classifies 
the bottom extension. This yields an exact sequence in $\cC$
\[ \xi : \quad 0 \longrightarrow U(A) 
\stackrel{\gamma - \iota}{\longrightarrow} U \times E(A) 
\longrightarrow G \longrightarrow 0. \]

\begin{proposition}\label{prop:comp}
Keep the above notation.

\begin{enumerate}

\item[{\rm (i)}] $\xi$ yields a projective resolution of 
$G$ in $\ucV$.

\item[{\rm (ii)}] For any object $H$ of $\cV$, we have an exact sequence
\[ 0 \longrightarrow \Hom_{\ucV}(G,H)
\stackrel{\varphi}{\longrightarrow} 
\Hom_{\cU}(U(G),U(H)) \times \Hom_{\ucA}(A(G),A(H)) \]
\[ \stackrel{\psi}{\longrightarrow} \Hom_{\cU}(U(A(G)),U(H))
\longrightarrow \Ext^1_{\ucV}(G,H)
\longrightarrow 0, \]
where $\varphi(\uf) := (U(\uf),\uA(\uf))$, and 
$\psi(u,v) := u \circ \gamma_G - \gamma_H \circ U(v)$.

\end{enumerate}

\end{proposition}

\begin{proof}
(i) This holds as $U,U(A)$ are projective in $\cC$ (Theorem 
\ref{thm:zero}), and $E(A)$ is projective in $\ucC$ (Proposition 
\ref{prop:vec}).

(ii) In view of (i), this follows readily from the long exact
sequence of extension groups 
\[ 0 \longrightarrow \Hom_{\ucV}(G,H) \longrightarrow 
\Hom_{\ucV}(U \times E(A),H) \longrightarrow \Hom_{\ucV}(U,H) 
\longrightarrow \Ext^1_{\ucV}(G,H) \longrightarrow 0 \]
associated with the short exact sequence $\xi$. 
\end{proof}

As a direct consequence of Proposition \ref{prop:comp}, we obtain:

\begin{corollary}\label{cor:orth}
The following conditions are equivalent for an object $G$ of $\cV$:

\begin{enumerate}

\item[{\rm (i)}] $G\cong E(A)$ in $\ucV$ for some abelian variety $A$.

 \item[{\rm (ii)}] $\Hom_{\ucV}(G,\bG_a) = \Ext^1_{\ucV}(G,\bG_a) = 0$. 

\end{enumerate}

\end{corollary}

As a further consequence, we describe the projective or injective 
objects of $\ucV$:

\begin{corollary}\label{cor:projinj}

\begin{enumerate}

\item[{\rm (i)}]
The projective objects of $\ucV$ are exactly the products
$V \times E(A)$, where $V$ is unipotent, and $A$ is an 
abelian variety.

\item[{\rm (ii)}]
The injective objects of $\ucV$ are exactly the abelian varieties.

\end{enumerate}

\end{corollary}

\begin{proof}
Let $G$ be an extension of an abelian variety $A$ by a unipotent
group $U$.

(i) If $G$ is projective in $\ucV$, then $\Ext^1_{\ucV}(G, \bG_a) = 0$. 
In view of Proposition \ref{prop:comp}, it follows that the map 
$\Hom_{\cU}(U,\bG_a) \to \Hom_{\cU}(U(A),\bG_a)$, 
$u \mapsto u \circ  \gamma$ is surjective. Equivalently,
$\gamma$ is injective; hence so is $\delta : E(A) \to G$.
Identifying $E(A)$ with a subgroup of $G$, it follows that 
$G = U + E(A)$, and $U(A) \subset U$. We may choose 
a complement $V \subset U$ to the subspace $U(A) \subset U$; 
then $G \cong V \times E(A)$. Conversely, every such product is
projective by Proposition \ref{prop:vec}. This yields the assertion.

(ii) If $G$ is injective in $\ucV$, then $\Ext^1_{\ucV}(B,G) = 0$
for any abelian variety $B$. Thus, we have an exact sequence
\[ 0 \longrightarrow \Hom_{\ucV}(B,U) \longrightarrow \Hom_{\ucV}(B,G) 
\longrightarrow \Hom_{\ucV}(B,A) 
\stackrel{\partial}{\longrightarrow} 
\Ext^1_{\ucV}(B,U) \longrightarrow 0. \]
Moreover, $\Ext^1_{\ucV}(B,U) \cong \Hom_{\cU}(U(B),U)$, as follows
e.g. from Proposition \ref{prop:comp}. Since 
$\Hom_{\ucV}(B,A) = \Hom_{\ucA}(B,A)$ is a finite-dimensional 
$\bQ$-vector space, so is $\Hom_{\cU}(U(B),U)$.

When $k$ is not a number field, i.e., $k$ is an infinite-dimensional
$\bQ$-vector space, this forces $U = 0$, since $U(B) \neq 0$ for
any non-zero abelian variety $B$. Thus, $G$ is an abelian variety.

On the other hand, when $k$ is a number field, there are only finitely
many isomorphism classes of abelian varieties that are isogenous to any 
prescribed abelian variety (see \cite{MW} for a quantitative version
of this finiteness result). 
As a consequence, we may choose a simple abelian variety $B$, 
not isogenous to any simple factor of $A$. Then $\Hom_{\ucA}(B,A) =0$;
as above, this yields $U = 0$, i.e., $G$ is an abelian variety. 

Conversely, every abelian variety is injective in $\ucV$ by Proposition 
\ref{prop:inj}. 
\end{proof}

We now describe the structure of $\cV$ and $\ucV$ in terms of 
linear algebra. Let $\cD$ be the category with objects the triples 
$(A,U,\gamma)$, where $A$ is an abelian variety, $U$ a unipotent
group, and $\gamma : U(A) \to U$ a morphism; the $\cD$-morphisms 
from $(A_1,U_1,\gamma_1)$ to $(A_2,U_2,\gamma_2)$ are those pairs
of $\cC$-morphisms $u: U_1 \to U_2$, $v: A_1 \to A_2$ such 
that the square
\[ \CD 
U(A_1) @>{U(v)}>> U(A_2) \\
@V{\gamma_1}VV @V{\gamma_2}VV \\
U_1 @>{u}>> U_2 \\
\endCD \]
commutes. We also introduce the `isogeny category' $\ucD$, by allowing
$v$ to be a $\ucC$-morphism in the above definition (this makes
sense in view of Corollary \ref{cor:vec}). Next, define a functor
\[ D : \cV \longrightarrow \cD \]
by assigning to each object $G$ the triple
$(A(G),R_u(G),\gamma)$, where $\gamma : U(A(G)) \to R_u(G)$
denotes the classifying map, and to each morphism $f : G_1 \to G_2$,
the pair $(A(f),U(f))$. By Corollary \ref{cor:vec} again, we may
define similarly a functor
\[ \uD : \ucV \longrightarrow \ucD. \]

\begin{proposition}\label{prop:vecequi}
With the above notation, the functors $D$ and $\uD$ yield
equivalences of categories.
\end{proposition}

We omit the easy proof.

\subsection{Semi-abelian varieties}
\label{subsec:sav}

Recall that a semi-abelian variety is an algebraic group $G$
that lies in an extension
\[ \xi : \quad 
0 \longrightarrow T \longrightarrow G \longrightarrow A
\longrightarrow 0, \]
where $T$ is a torus, and $A$ an abelian variety. We now adapt part 
of the results of \S \ref{subsec:veav} to this setting, leaving the
(easy) verifications to the motivated reader. The algebraic groups 
$T = T(G)$ and $A = A(G)$ are uniquely determined
by $G$, and the extension $\xi$ has no non-trivial automorphisms.
Thus, the data of $G$ and of the extension $\xi$ are equivalent.
Moreover, recall the natural isomorphism
\[ c :  \Ext^1_{\cC}(A,T) \stackrel{\cong}{\longrightarrow}
\Hom^{\Gamma}(X(T), \wA(k_s)), \]
which arises from the Weil-Barsotti isomorphism
\[ \Ext^1_{\cC_{k_s}}(A_{k_s},\bG_{m,k_s}) 
\stackrel{\cong}{\longrightarrow} \wA(k_s) \]
together with the pairing 
\[ \Ext^1_{\cC}(A,T) \times X(T) \longrightarrow
\Ext^1_{\cC_{k_s}}(A_{k_s},\bG_{m,k_s})
\]
given by push-out of extensions via characters of $T$.

Denote by $\cS$ the full subcategory of $\cC$ with objects the 
semi-abelian varieties. Then the analogue of Lemma \ref{lem:vec}
holds in view e.g. of \cite[\S 5.4]{Brion-II} (but there is no analogue
of the universal vector extension in this setting). Thus, the isogeny
category of semi-abelian varieties, $\ucS$, is an abelian category. 
As for vector extensions of abelian varieties, we have natural
isomorphisms
\[ \bQ \otimes_\bZ \Hom_{\cC}(G_1,G_2) 
\stackrel{\cong}{\longrightarrow} \Hom_{\ucS}(G_1,G_2), \quad
\bQ \otimes_\bZ \Ext^1_{\cC}(G_1,G_2) 
\stackrel{\cong}{\longrightarrow} 
\Ext^1_{\ucS}(G_1,G_2) \]
for any objects $G_1,G_2$ of $\ucS$. This yields a natural isomorphism
\[ \Ext^1_{\ucC}(A,T) \stackrel{\cong}{\longrightarrow}
\Hom^{\Gamma}(X(T)_{\bQ}, \wA(k_s)_{\bQ}). \]
Note that the assignment $A \mapsto \wA(k_s)_{\bQ}$ yields an
exact functor from $\ucA$ to the category of $\bQ$-vector spaces
equipped with the discrete topology and a continuous representation 
of $\Gamma$, as follows e.g. from Remarks \ref{rem:av}.

Next, we obtain a description of $\ucS$ in terms of linear algebra. 
Let $\ucE$ be the category with objects the triples
$(A,M,c)$, where $A$ is an abelian variety, $M$ a finite-dimensional
$\bQ$-vector space equipped with a continuous action of $\Gamma$, and 
$c : M \to \wA(k_s)_{\bQ}$ a $\Gamma$-equivariant linear map;
the $\ucE$-morphisms from $(A_1,M_1,c_1)$ to $(A_2,M_2,c_2)$
are those pairs $(\uu,v)$, where $\uu: A_1 \to A_2$
is a $\ucA$-morphism and $v: M_2 \to M_1$ a $\Gamma$-equivariant
linear map, such that the square
\[ \CD
M_2 @>{c_2}>> \wA_2(k_s)_{\bQ} \\
@V{v}VV @V{\hat{\uu}}VV \\
M_1 @>{c_1}>> \wA_1(k_s)_{\bQ} \\
\endCD \]
commutes. Then one may check that the assignment 
$G \mapsto (A(G), X(T(G))_{\bQ}, c(G)_{\bQ})$ 
yields an equivalence of categories $\ucS \to \ucE$. 
Moreover, the sequence
\[ 0 \longrightarrow \Hom_{\ucS}(G_1,G_2)
\stackrel{\varphi}{\longrightarrow} 
\Hom^{\Gamma}(M_2,M_1) \times \Hom_{\ucA}(A_1,A_2) \]
\[ \stackrel{\psi}{\longrightarrow} \Hom^{\Gamma}(M_2,\wA_1(k_s)_\bQ)
\longrightarrow \Ext^1_{\ucS}(G_1,G_2)
\longrightarrow 0 \]
turns out to be exact for any semi-abelian varieties $G_1,G_2$, 
where $\ucE(G_i) := (A_i,M_i,c_i)$ for $i = 1,2$, 
$\varphi(\uf) := (\uX_{\bQ} \circ \uT)(\uf),\uA(\uf))$, and
$\psi(u,v) := c_1 \circ u - \hat{v} \circ c_2$.

Yet there are important differences between the isogeny categories
of vector extensions and semi-abelian varieties. For example, the
latter does not have enough projectives in general:

\begin{proposition}\label{prop:sabproj}

\begin{enumerate}

\item[{\rm (i)}]
If $k$ is not locally finite, then the projective objects of 
$\ucS$ are exactly the tori.

\item[{\rm (ii)}]
If $k$ is locally finite, then the product functor
$\ucT \times \ucA \to \ucS$ yields an equivalence of categories.

\end{enumerate}

\end{proposition}

\begin{proof}
(i) Let $G$ be a semi-abelian variety, extension of an abelian
variety $A$ by a torus $T$. Denote by $c : X(T) \to \wA(k_s)$
the classifying map, and by 
\[ c_\bQ : X(T)_\bQ \longrightarrow \wA(k_s)_\bQ \]
the corresponding $\bQ$-linear map; recall that $c$ and $c_\bQ$ are 
$\Gamma$-equivariant.
 
If $G$ is projective in $\ucC$, then 
$\Ext^1_{\ucC}(G,T') = 0$ for any torus $T'$. Thus, we have
an exact sequence
\[ 0 \longrightarrow \Hom_{\ucC}(A,T') \longrightarrow \Hom_{\ucC}(G,T') 
\longrightarrow \Hom_{\ucC}(T,T') 
\stackrel{\partial}{\longrightarrow} 
\Ext^1_{\ucC}(A,T') \longrightarrow 0. \]
Moreover, the boundary map $\partial$ may be identified with the map  
\[ \Hom^{\Gamma}(X(T')_\bQ,X(T)_\bQ) \longrightarrow 
\Hom^{\Gamma}(X(T')_\bQ, \wA(k_s)_\bQ). \quad
f \longmapsto c_\bQ \circ f. \] 
Since $\partial$ is surjective, and $X(T')_{\bQ}$ may be chosen 
arbitrarily among finite-dimensional $\bQ$-vector spaces equipped 
with a continuous representation of $\Gamma$, the map $c_\bQ$ 
is surjective as well. In particular, the abelian group $\wA(k_s)$
has finite rank. In view of \cite[Thm.~9.1]{FJ}, this forces 
$A$ to be zero, i.e., $G$ is a torus.

(ii) This follows readily from Proposition \ref{prop:hom} 
and Corollary \ref{cor:pos}.
\end{proof}

\subsection{Product decompositions}
\label{subsec:pd}

In this subsection, we first prove the remaining assertions (ii) and (iii)
of Theorem \ref{thm:all}. Then we describe the isogeny category $\ucC$ 
in characteristic $0$, and its projective or injective objects in arbitrary 
characteristics.

\begin{proposition}\label{prop:prodlin}

\begin{enumerate}

\item[{\rm (i)}] If $k$ is perfect, then the product functor 
$\cM \times \cU \to \cL$ yields an equivalence of categories.

\item[{\rm (ii)}] For any field $k$, the product functor
$\ucT \times \ucU \to \ucL$ yields an equivalence of categories.

\end{enumerate}

\end{proposition}

\begin{proof}
(i) This follows readily from Theorem \ref{thm:lin} and
Proposition \ref{prop:hom}.

(ii) This is a consequence of (i) in view of Theorem \ref{thm:insep}.
\end{proof}

\begin{proposition}\label{prop:prodpos}
If $\car(k) > 0$, then the product functor $\ucS \times \ucU \to \ucC$ 
yields an equivalence of categories.
\end{proposition}

\begin{proof}
Let $G$ be a smooth connected algebraic group. By Corollary
\ref{cor:pos}, there exists a finite subgroup $F \subset G$ such
that $G/F \cong S \times U$, where $S$ is a semi-abelian variety,
and $U$ is unipotent. Thus, the product functor $\Pi$ is essentially 
surjective.

Next, let $S_1,S_2$ be semi-abelian varieties, and $U_1,U_2$ smooth
connected unipotent groups. We check that $\Pi$ induces an
isomorphism
\[ \Hom_{\ucC}(S_1, S_2) \times \Hom_{\ucC}(U_1, U_2) \longrightarrow
\Hom_{\ucC}(S_1 \times U_1, S_2 \times U_2), \quad
(\uvarphi,\upsi) \longmapsto \uvarphi \times \upsi. \]

Assume that $\uvarphi \times \upsi = 0$. Choose representatives
$\varphi: S_1 \to S_2/S'_2$, $\psi: U_1 \to U_2/U'_2$, where
$S'_2,U'_2$ are finite. Then 
$\varphi \times \psi : S_1 \times U_1 
\to (S_2 \times U_2)/(S'_2 \times U'_2)$
has finite image, and hence is zero by Lemma \ref{lem:epi}.
So $\uvarphi = \upsi = 0$.

Let $\ugamma \in \Hom_{\ucC}(S_1 \times U_1, S_2 \times U_2)$
be represented by $\gamma : S_1 \times U_1 \to (S_2 \times U_2)/F$,
where $F$ is finite. Then $F \subset S'_2 \times U'_2$ for some 
finite subgroups $S'_2 \subset S_2$, $U'_2 \subset U_2$. Thus, we may
assume that $F = S'_2 \times U'_2$. Then the composite morphisms
\[ S_1 \longrightarrow S_1 \times U_1
\stackrel{\gamma}{\longrightarrow} S_2/S'_2 \times U_2/U'_2 
\longrightarrow U_2/U'_2, \]
\[ U_1 \longrightarrow S_1 \times U_1
\stackrel{\gamma}{\longrightarrow} S_2/S'_2 \times U_2/U'_2 
\longrightarrow S_2/S'_2 \]
are zero by Lemma \ref{lem:epi} and Proposition \ref{prop:hom}.
Thus, $\gamma = \varphi \times \psi$ for some morphisms
$\varphi : S_1 \to S_2/S'_2$, $\psi: U_1 \to U_2/U'_2$.
\end{proof}

Combining Propositions \ref{prop:sabproj} (i) and \ref{prop:prodpos}, 
we obtain readily:

\begin{corollary}\label{cor:prodlocf}
If $k$ is locally finite, then the product functor
\[ \ucT \times \ucA \times \ucU \longrightarrow \ucC \] 
yields an equivalence of categories.
\end{corollary}

\begin{remarks}\label{rem:prodlocf}
(i) With the notation of the above corollary, each of the categories
$\ucT$, $\ucA$, $\ucU$ admits a description of its own. By Proposition 
\ref{prop:tss}, $\ucT$ is equivalent to the category of 
$\bQ$-vector spaces equipped with an automorphism of finite order. 
Also, the isomorphism classes of abelian varieties over a finite 
field are classified by the Honda-Tate theorem (see 
\cite{Honda, Tate}); their endomorphism rings are investigated in 
\cite{Waterhouse}. Finally, the structure of $\ucU$ (obtained
in \cite[\S V.3.6.7]{DG}) has been described in Remark \ref{rem:fe}.

(ii) Combining Lemma \ref{lem:fin}, Theorem \ref{thm:lin} and 
Lemma \ref{lem:des}, one may show that the product functor 
$\cM/\cI \cM \times \cU/ \cI \cU \to \cL/\cI$
yields an equivalence of categories. Here $\cI$ denotes the
category of infinitesimal algebraic groups, and $\cI \cM$
(resp.~$\cI \cU$) the full subcategory of infinitesimal groups 
of multiplicative type (resp.~unipotent).
\end{remarks}

Next, assume that $\car(k) = 0$. Then every algebraic group
is isogenous to a fibered product $E \times_A S$, where
$E$ is a vector extension of the abelian variety $A$, and
$S$ is semi-abelian with Albanese variety isomorphic to $A$ 
(see e.g. Remark \ref{rem:zero}). This motivates the consideration of
the fibered product $\ucV \times_{\ucA} \ucS$: this is the category
with objects the triples $(E,S,\uf)$, where $E$ is a vector extension
of an abelian variety, $S$ a semi-abelian variety, and 
$\uf: A(E) \to A(S)$ an $\ucA$-isomorphism. The morphisms from
$(E_1,S_1,\uf_1)$ to $(E_2,S_2,\uf_2)$ are those pairs of 
$\ucC$-morphisms $\uu: E_1 \to E_2$, $\uv: S_1 \to S_2$ such that
the square 
\[ \CD
A(E_1) @>{\uA(\uu)}>> A(E_2) \\
@V{\uf_1}VV @V{\uf_2}VV \\
A(S_1) @>{\uA(\uv)}>> A(S_2) \\
\endCD \]
commutes in $\ucA$. 

\begin{proposition}\label{prop:prodzero}
If $\car(k) = 0$, then $\ucC$ is equivalent to 
$\ucV \times_{\ucA} \ucS$.
\end{proposition}

The proof is similar to that of Proposition \ref{prop:prodpos}, 
and will be omitted. Note that the descriptions of $\ucV$ and $\ucS$ 
in terms of linear algebra, obtained in \S \ref{subsec:veav} and
\S \ref{subsec:sav}, can also be reformulated in terms of 
fibered products of categories.

Returning to an arbitrary field $k$, we obtain:

\begin{theorem}\label{thm:proj}
The projective objects of $\ucC$ are exactly:

\begin{itemize}

\item the linear algebraic groups, if $\car(k) = 0$.

\item the semi-abelian varieties, if $k$ is locally finite.

\item the tori, if $\car(k) >0$ and $k$ is not locally finite.

\end{itemize}

\end{theorem}

\begin{proof}
Let $G$ be a smooth connected algebraic group. As a consequence
of Theorem \ref{thm:zero} and Proposition \ref{prop:prodpos}, 
we have an exact sequence in $\ucC$
\[ 0 \longrightarrow U \longrightarrow G \longrightarrow S 
\longrightarrow 0, \]
where $U$ is smooth, connected, and unipotent, and $S$ is a
semi-abelian variety.

If $G$ is projective in $\ucC$, then $\Ext^1_{\ucC}(G,T') = 0$
for any torus $T'$. Since $\Hom_{\ucC}(U,T') = 0$ (as a consequence 
of Proposition \ref{prop:hom}) and $\Ext^1_{\ucC}(U,T') = 0$ 
(by Lemma \ref{lem:list}), the long exact sequence for Ext groups 
yields that $\Ext^1_{\ucC}(S,T') = 0$ as well. By arguing as in 
the proof of Proposition \ref{prop:sabproj}, this forces
either $A(S)$ to be zero, or $k$ to be locally finite.

If $A(S) = 0$, then $G$ is linear, and hence $G \cong T \times U$
in $\ucC$ by Proposition \ref{prop:prodlin}. Moreover, tori are 
projective in $\ucC$ by Proposition \ref{prop:proj}; thus, we may 
assume that $G$ is unipotent. If $\car(k) = 0$, then every unipotent
group is projective, as follows e.g. from Lemma \ref{lem:list}. 
If $\car(k) > 0$ and $G \neq 0$, then there exists an exact sequence 
\[ 0 \longrightarrow H \longrightarrow G \longrightarrow \bG_a 
\longrightarrow 0 \] 
in $\cC$. Since $\Ext^1_{\ucC}(G,\bG_a) = 0 = \Ext^2_{\ucC}(H,\bG_a)$, 
it follows that $\Ext^1_{\ucC}(\bG_a,\bG_a) = 0$. But this contradicts 
Example \ref{ex:nonsplit} (ii), hence $G = 0$.

On the other hand, if $k$ is locally finite, then 
$G \cong T \times A \times U$ in $\ucC$
(by Corollary \ref{cor:prodlocf}) and it follows as above
that $U$ is zero. Conversely, every semi-abelian variety
is projective in $\ucC$, by Corollary \ref{cor:prodlocf} again.
\end{proof}

\begin{corollary}\label{cor:proj}
If $\car(k) > 0$, then $\cC$ has no non-zero projective objects.
\end{corollary}

\begin{proof}
Let $G$ be a projective object of $\cC$. By the claim in the proof of
Theorem \ref{thm:zero} (iii), the (abstract) group $\Hom_{\cC}(G,H)$ is 
divisible for any divisible group $H$. By arguing as in that proof, it
follows that $G$ is linear. Also, for any torus $T$, the group
$\Hom_{\cC}(G,T)$ is finitely generated and divisible, hence zero.

Next, we show that $G$ is connected. Indeed,  the quotient $G/G^0$ is 
finite and \'etale, hence contained in a torus $T$.  As $\Hom_{\cC}(G,T) = 0$, 
this yields the assertion.

In view of Theorem \ref{thm:lin}, we obtain an exact sequence
\[ 0 \longrightarrow M \longrightarrow G \longrightarrow U  
\longrightarrow 0, \] 
where $M$ is of multiplicative type, and $U$ is unipotent and
connected. We now show that $U$ is projective in the unipotent
category $\cU$. Indeed, given an exact sequence 
$U_1 \to U_2 \to 0$ in $\cU$ and a morphism $\varphi : U \to U_2$,
we may lift the composition $\psi : G \to U \to U_2$ to a morphism
$\gamma : G \to U_1$. Then $\gamma(M) = 0$ (Proposition \ref{prop:hom})
and hence $\gamma$ factors through a morphism $\delta: U \to U_1$, 
which lifts $\varphi$.
 
Let $F$ be an infinitesimal subgroup of $U$ such that $U/F$ is smooth; then
$U/F$ is an object of $\cU$, and one easily checks by using Proposition
\ref{prop:exa} that $U/F$ is projective in $\ucU$. In view of Proposition
\ref{prop:prodpos} and Theorem \ref{thm:proj}, it follows that $U/F$ is zero,
hence $U$ is infinitesimal. If $U \neq 0$ then there exists an epimorphism
$U \to \alpha_p$ and hence a non-zero morphism $U \to A$ for some
(supersingular) abelian variety $A$. Since $\Hom_{\cC}(G,A) = 0$,
this yields a contradiction. Thus, $U = 0$, i.e., $G$ is of multiplicative 
type, hence contained in a torus, hence zero.
\end{proof}

\begin{theorem}\label{thm:inj}
The injective objects of $\ucC$ are exactly the semi-abelian varieties
if $k$ is locally finite, and the abelian varieties otherwise.
\end{theorem}

\begin{proof}
Let $G$ be an injective object of $\ucC$, and  
$0 \to G_1 \to G \to G_2 \to 0$ 
an exact sequence in $\ucC$. Since $\hd(\ucC) = 1$, the natural
map $\Ext^1_{\ucC}(H, G) \to \Ext^1_{\ucC}(H,G_2)$ is an epimorphism
for any object $H$ of $\ucC$. Thus, every quotient of $G$ is injective.
In particular, so is the largest semi-abelian quotient $S$ (Remark
\ref{rem:zero} and Theorem \ref{thm:pos}).

We now adapt the argument of Corollary \ref{cor:projinj}. Recall that
 $S$ lies in a unique extension 
\[ 0 \longrightarrow T \longrightarrow S \longrightarrow A \longrightarrow 0, \]
classified by a $\Gamma$-equivariant morphism 
\[ c : X(T) \longrightarrow \wA(k_s). \]
Since $S$ is injective in $\ucC$, we have $\Ext^1_{\ucC}(B,S) = 0$ 
for any abelian variety $B$, and hence the connecting homomorphism
\[ \partial : \Hom_{\ucC}(B,A) \longrightarrow \Ext^1_{\ucC}(B,T) \]
is surjective. Under the isomorphisms 
\[ \Hom_{\ucC}(B,A) \cong \bQ \otimes_{\bZ} \Hom_{\cC}(B,A), \quad 
\Ext^1_{\ucC}(B,T) \cong 
\bQ \otimes_{\bZ} \Hom^{\Gamma}(X(T),\wB(k_s)) \]
(Proposition \ref{prop:homdiv} and \S \ref{subsec:sav}), the map $\partial$ 
is identified with the composition of the natural map
\[ \bQ \otimes_{\bZ} \Hom_{\cC}(B,A) \longrightarrow 
\bQ \otimes_{\bZ} \Hom^{\Gamma}(\wA(k_s),\wB(k_s)) \]
with the map
\[ \gamma : \bQ \otimes_{\bZ} \Hom^{\Gamma}(\wA(k_s),\wB(k_s)) 
\longrightarrow \bQ \otimes_{\bZ} \Hom^{\Gamma}(X(T),\wB(k_s)),
\quad u \longmapsto u \circ c. \]

Next, consider a free abelian group $M$ of finite rank, equipped with
a continuous action of $\Gamma$. Then for any abelian variety $B$,
the tensor product $B_{k_s} \otimes_{\bZ} M$ is a $k_s$-abelian variety 
equipped with the diagonal $\Gamma$-action, and descends to a unique 
$k$-abelian variety $B(M)$. Replacing $B$ with $B(M)$, it follows that the
corresponding map 
\[ \partial(M) : 
\bQ \otimes_{\bZ} \Hom^{\Gamma}_{\cC_{k_s}}
(B_{k_s} \otimes_{\bZ} M, A_{k_s}) 
\longrightarrow 
\bQ \otimes_{\bZ} \Hom^{\Gamma}(X(T) \otimes_{\bZ} M,\wB(k_s)) \]
is surjective as well. As 
$\bQ \otimes_{\bZ} \Hom_{\cC_{k_s}}(B_{k_s},A_{k_s})$ and
$\bQ \otimes_{\bZ} \Hom(X(T),\wB(k_s))$ are direct sums of continuous,
finite-dimensional $\Gamma$-modules, and $M$ is arbitrary, it follows that 
the natural map
\[ \Hom_{\cC_{k_s}}(B_{k_s},A_{k_s}) \longrightarrow 
\Hom(X(T),\wB(k_s)) \]
is surjective over the rationals.

In the case where $k$ is not locally finite, recall from 
\cite[Thm.~9.1]{FJ} that the group $\wB(k_s)$ 
has infinite rank for any non-zero abelian variety $B$. Since 
$\Hom_{\cC_{k_s}}(B_{k_s},A_{k_s})$ has finite rank, it follows that
$X(T) = 0$, i.e., $S$ is an abelian variety. We now distinguish between
two subcases.

If $\car(k) = 0$, then $G$ is a vector extension of $A$, and hence is
injective in the category $\ucV$. Thus, $G = A$ in view of Corollary
\ref{cor:projinj}.

If $\car(k) > 0$, then $G \cong A \times U$ in $\ucC$ for some split 
unipotent group $U$, which must be injective in $\ucC$. If in addition
$U \neq 0$, then it follows that $\bG_a$ (a quotient of $U$) is injective
in $\ucC$ as well. But $\Ext^1_{\ucC}(\bG_a, \bG_a) \neq 0$ in view
of Example \ref{ex:nonsplit}, a contradiction. Thus, $G = A$ again.

Finally, in the case where $k$ is locally finite, we have 
$G \cong T \times A \times U$ in $\ucC$ with an obvious notation.
As above, we obtain that $U = 0$; on the other hand, $T$ and
$A$ are injective in $\ucC$ by Corollary \ref{cor:prodlocf}. This completes
the proof.
\end{proof}

\begin{remark}\label{rem:inj}
The category $\cC$ has no non-zero injective objects. This result
should be well-known, but we could not find it in the literature;
also, it does not seem to follow from Theorem \ref{thm:inj}, as the
relation between injective objects in $\cC$ and $\ucC$ is unclear.
So we sketch a direct proof: let $G$ be an injective object of $\cC$.
For any positive integer $n$ which is prime to $\car(k)$, the 
$n$-torsion subgroup $G[n]$ is finite and \'etale, hence isomorphic
to a subgroup of some torus $T_n$. The inclusion $G[n] \subset G$
extends to a morphism $T_n \to G$, and hence $G[n]$ is contained
in the maximal torus $T(G)$. Likewise, $G[n]$ is isomorphic to
a subgroup of an abelian variety $A_n$, and hence is contained
in the largest abelian subvariety of $G$. Since this abelian
variety intersects $T(G)$ along a finite subgroup, it follows 
that $G[n] = 0$ for $n \gg 0$. 

If $\car(k) = 0$, then there is an exact sequence
$0 \to R_u(G) \to G \to H \to 0$, where $H^0$ is a semi-abelian
variety. Since $n_{R_u(G)}$ is an isomorphism, we have 
$G[n] \cong H[n]$ for all $n$; it follows that $H$ is finite.
Thus, $G = R_u(G) + F$ for some finite group $F$, which embeds into
some torus $T$. As above, it follows that $F = 0$, i.e., $G$ is
unipotent. If $G \neq 0$, then $\Ext^1_{\cC}(A,G) \neq 0$ for any
abelian variety $A$, a contradiction.

If $\car(k) = p > 0$, then there is an exact sequence
$0 \to H \to G \to U \to 0$, where $H^0$ is again a semi-abelian
variety, $H/H^0$ is of multiplicative type, and $U$ is unipotent. 
Thus, $H[n] \cong G[n]$ for all $n$ prime to $p$, and hence $H$ 
is finite. It follows as above that $H = 0$, i.e., $G$ is unipotent.
If $G \neq 0$, then $G$ contains a copy of $\alpha_p$, which embeds 
into some abelian variety; this yields a contradiction.
\end{remark}

\medskip

\noindent
{\bf Acknowledgements}.
I warmly thank Claire Amiot, Brian Conrad, H\'el\`ene Esnault, 
St\'ephane Guillermou, Bruno Kahn, Giancarlo Lucchini Arteche, 
Ga\"el R\'emond, Jean-Pierre Serre, and Catharina Stroppel for 
very helpful discussions or e-mail exchanges. Special thanks 
are due to an anonymous referee for a careful reading and 
valuable comments.

\bibliographystyle{amsalpha}

\end{document}